\documentclass[reqno,11pt]{amsart} 

\usepackage{amsmath, amsthm, a4, latexsym, amssymb, pstricks, setspace, scrtime}
\usepackage[numbers]{natbib}
\usepackage{indentfirst}
\usepackage{graphicx}
\usepackage{amsfonts,xcolor}
\usepackage{tikz}
\usepackage{pgfplots}
\usetikzlibrary{patterns}
\usetikzlibrary{plothandlers}
\usetikzlibrary{plotmarks}
\usepgflibrary{plothandlers}
\usepgflibrary{plotmarks}
\usetikzlibrary{decorations.pathreplacing}
\usepackage{etoolbox}
\usepackage{bbm}

\def\@rmrk#1#2{\refstepcounter
    {#1}\@ifnextchar[{\@yrmrk{#1}{#2}}{\@xrmrk{#1}{#2}}}

\makeatletter\@addtoreset{equation}{section}\makeatother

 \newfont{\bfit}{cmbxti10 scaled 1200}

 \newcommand{\eps}{\varepsilon}

 \newcommand{\R}{\mathbb{R}}
 \newcommand{\N}{\mathbb{N}}
 \newcommand{\Z}{\mathbb{Z}}
 
 \newcommand{\prob}{\mathbb{P}}

 \newcommand{\Q}{\mathbb{Q}}
 \newcommand{\1}{{\sf 1}}

 \newcommand{\skris}{{\mathcal S}}

 \newcommand{\heap}[2]{\genfrac{}{}{0pt}{}{#1}{#2}}
 \newcommand{\sfrac}[2]{\mbox{$\frac{#1}{#2}$}}
 
 \newcommand{\ssup}[1] {{\scriptscriptstyle{({#1}})}}

\renewcommand{\subsection}{\secdef \subsct\sbsect}
\newcommand{\subsct}[2][default]{\refstepcounter{subsection}
\vspace{0.15cm}
{\flushleft\bf \arabic{section}.\arabic{subsection}~\bf #1  }
\nopagebreak\nopagebreak}
\newcommand{\sbsect}[1]{\vspace{0.1cm}\noindent
{\bf #1}\vspace{0.1cm}}

\newenvironment{example}{\refstepcounter{theorem}
{\bf Example \thetheorem\ }\nopagebreak  }%
{\nopagebreak {\hfill\rule{2mm}{2mm}}
\\ }

\newtheorem{theorem}{Theorem}[section]
\newtheorem{thm}{Theorem}
\newtheorem{lemma}[theorem]{Lemma}

\newtheorem{prop}[theorem]{Proposition}

\newcommand{\p}{\mathbb{P}}

\newcommand{\E}{\mathbb{E}}

\newcommand{\s}{\mathcal{S}}

\newtheoremstyle{thm}{1.5ex}{1.5ex}{\itshape\rmfamily}{}
{\bfseries\rmfamily}{}{2ex}{}

\newtheoremstyle{rem}{1.3ex}{1.3ex}{\rmfamily}{}
{\itshape\rmfamily}{}{1.5ex}{}
\theoremstyle{rem}

\refstepcounter{subsubsection}

\def\thebibliography#1{\section*{References}
  \list%
  {\arabic{enumi}.}
    {\settowidth\labelwidth{[#1]}\leftmargin\labelwidth
    \advance\leftmargin\labelsep
    \parsep0pt\itemsep0pt
    \usecounter{enumi}}
    \def\newblock{\hskip .11em plus .33em minus .07em}
    \sloppy                   
    \sfcode`\.=1000\relax}

\begin{document}

\title[Skorokhod embeddings for two-sided Markov chains]
{Skorokhod embeddings for two-sided Markov chains}

\centerline{\LARGE \bf Skorokhod embeddings for two-sided Markov chains}

\vspace{0.4cm}

\thispagestyle{empty}
\vspace{0.3cm}
\textsc{Peter M\"orters\footnote{Communicating author.}} and \textsc{Istv\'an Redl}\\
Department of Mathematical Sciences,
University of Bath, Bath BA2 7AY, England\\
E--mail: \texttt{maspm@bath.ac.uk} and \texttt{ir250@bath.ac.uk}

\vspace{0.3cm}


\begin{quote}{\small {\bf Abstract: } Let \smash{$(X_n \colon n\in\Z)$} be a two-sided recurrent Markov chain
with fixed initial state~$X_0$ and let $\nu$ be a probability measure on its state space. We give a
necessary and sufficient criterion for the existence of a non-randomized time~$T$ such that
\smash{$(X_{T+n} \colon n\in\Z)$} has the law of the same Markov chain with initial distribution~$\nu$.
In the case when our criterion is satisfied we give an explicit solution, which is also a stopping time, and
study its moment properties. We show that this solution minimizes the expectation of $\psi(T)$  in the class
of all non-negative solutions, simultaneously for  all non-negative concave functions $\psi$.
}
\end{quote}
\vspace{0.4cm}

{\footnotesize
{\bf MSc classification (2000):} Primary 60J10 Secondary  60G40, 05C70.\\[-5mm]

}



\section{Introduction and statement of main results}\label{S:Intro}

Let $\mathcal{S}$ be a finite or countable state space and $p=(p_{ij} \colon i,j\in\s)$ an irreducible and recurrent transition matrix. Then there exists a stationary measure~$(m_i \colon i\in \s)$ with positive weights, which is finite in the positive
recurrent case, and infinite otherwise. The two-sided stationary Markov chain $X=(X_{n} \colon n \in \Z)$ with initial measure $(m_i \colon i \in\s)$ and transition
matrix~$p$ is characterized by
\begin{itemize}
\item $\p(X_n=i)=m_i$ for all $n\in\Z$, $i\in\s$;\\[-3mm]
\item $\prob(X_n=j \, | X_{n-1}, X_{n-2},\ldots)=p_{X_{n-1}j}$ for all $n\in\Z$, $i,j\in\s$.
\end{itemize}
This chain always exists, if we allow $\p$ to be a $\sigma$-finite measure. For the simplest construction, let
$(X_{n} \colon n \ge 0)$ be the chain with initial measure $(m_i \colon i \in\s)$ and transition matrix~$p$, and  $(X_{-n} \colon n \ge 0)$ be the chain
with given initial state $X_0$ and dual transition probabilities given by $p^*_{ij}=(m_j/m_i)p_{ji}$.%
\medskip%

By conditioning the stationary chain~$X$ on the event $\{X_0=i\}$, we define the two-sided Markov chain  with transition
matrix~$p$ with fixed initial state~$X_0=i$. Its law, denoted by~$\p_i$, does not depend on the choice of~$(m_i \colon i\in \s)$ and
is  always a probability law. Note that we can equivalently define this chain, or indeed the two-sided Markov chain  with transition
matrix~$p$ and arbitrary  initial distribution $\nu$, by picking $X_0$ according to $\nu$ and letting the forward and backward chains
$(X_{n} \colon n \ge 0)$, resp.~$(X_{-n} \colon n \ge 0)$, evolve as in the case of the stationary chain.
\medskip

A natural version of the \emph{Skorokhod embedding problem} in this context asks, given
the two-sided Markov chain $(X_{n} \colon n \in \Z)$ with transition matrix~$p$ and initial state~$X_0=i$
and a  probability measure $\nu$ on the state space~$\s$, whether there exists a random time~$T$ such that $(X_{n+T} \colon n \in \Z)$ is a two-sided Markov chain with transition matrix $p$ such that $X_{T}$ has law $\nu$. If this is the case we say that $T$ is an \emph{embedding} of  the \emph{target distribution}~$\nu$.
Our interest here is mainly in times $T$ which are \emph{non-randomized}, which means that $T$~is a measurable function of the sample chain~$X$. The random times~$T$ are often stopping times, but this is not a necessary requirement.
\medskip

\pagebreak[3]

Finding embeddings of two-sided Markov chains is a subtle problem, because even for stopping
times~$T$ the shifted process  \smash{$T^{-1}X:=(X_{n+T} \colon n \in \Z)$} often will \emph{not} be a
two-sided Markov chain.  For example, take a simple symmetric random walk on the integers, started in $X_{0}=0$,
and let $T$ be the first positive hitting time of the integer~$a>0$. Then~$T$ embeds the Dirac measure~$\delta_a$, but the increment $T^{-1}X_{0}-T^{-1}X_{-1}$ always takes  the value $+1$, hence  $T^{-1}X$  is not a two-sided simple random walk.
A similar argument shows that even shifting the simple random walk by a nonzero fixed time does not preserve the
property of being a simple random walk with given distribution of the state at time zero.%
\medskip%

The first main result of this paper gives a necessary and sufficient condition on the initial state, the target measure and the stationary distribution
for the existence of a Skorokhod embedding for an arbitrary two-sided Markov chain.
\smallskip

\begin{thm}\label{T:characterization_thm_intro}
Let $X$ be a two-sided irreducible and recurrent Markov chain with transition matrix~$p$ and initial state~$X_{0}=i$. Take $\nu=(\nu_j \colon j\in\s)$ to be
any  probability measure on $\mathcal{S}$. Then the following statements are equivalent.
\begin{itemize}
\item[$(a)$] There exist a non-randomized random time $T$ such that $(X_{n+T} \colon n \in \Z)$ is a Markov chain with transition
matrix $p$ and $X_{T}$ has law $\nu$.  \\[-2mm]
\item[$(b)$] The stationary measure $(m_{j} \colon j \in \mathcal{S})$ satisfies
$\frac{m_{i}}{m_{j}}\,\nu_j \in \Z$ 
for all $j \in \s$.
\end{itemize}
If the random time~$T$ in $(a)$ exists it can always be taken to be a stopping time.
\end{thm}
\smallskip

\begin{example}(Embedding measures with mass in the initial state)\label{Ex:characterization_thm_intro}
Assume that the target measure~$\nu$ charges the initial state~$i \in \s$ of the Markov chain, i.e.\ $\nu_i>0$.
Choosing $i=j$ in $(b)$ shows that a non-randomized random time~$T$ with the properties of $(a)$ can exist only if
$\nu = \delta_{i}$. In this case a natural family of embeddings can be constructed 
using the concept of point stationarity, see for example~\cite{Th00},
as follows: Let $r\in\N$ and let $T_r$ be the the time of the $r$th
visit of state~$i$ after time zero. Then it is easy to check, and follows from~\cite[Theorem 6.3]{LT09},
that the process $T_r^{-1}X$  is a Markov chain with transition matrix~$p$ and $X_{T_r}=i$.
\end{example}\ \\[-6mm]

\begin{example}(Extra head problem)\label{Ex:extra_head}
Take a doubly-infinite sequence of tosses of a (possibly biased) coin, or more precisely
let $X=(X_n \colon n\in\Z)$ be i.i.d.\ random variables with distribution
$\p(X_{n}={\texttt{head}})=p$, $\p(X_{n}={\texttt{tail}})=1-p$, for some $p\in(0,1)$.
Our aim is to find, without using any randomness generated in a way different from looking
at coins in the sequence, a coin showing head in this sequence in such a way that the two semi-infinite sequences of coins
to the left and to the right of this coin  remain independent i.i.d.\ sequences of coins with the same bias. This is known
as \emph{extra head problem} and was investigated and fully answered by Liggett~\cite{LT02}
and Holroyd and Peres~\cite{HP05}. To relate this to our setup, we can assume that $X_0={\texttt{tail}}$, as otherwise the coin at the origin is the extra head.
Then the extra head problem becomes the Skorokhod embedding problem for~$X$ with initial
state~$X_0= {\texttt{tail}}$ and target measure $\nu=\delta_{{\texttt{head}}}$. Theorem~\ref{T:characterization_thm_intro}
shows (as proved by Holroyd and Peres before) that the extra head problem has a solution if and only if
\smash{$(1-p)/p\in\Z$}, i.e.\ if and only if $p$ is the inverse of an integer. Moreover, Liggett~\cite{LT02} gives an explicit
solution of the extra head problem which we generalize to our setup in Theorem~2 below.
\end{example}\ \\[-6mm]

\begin{example}(Inverse extra head problem)\label{Ex:inv_extra_head}
If in the setup of Example~\ref{Ex:extra_head} the state of the coin at the origin has been revealed, we ask
whether it is possible to shift the sequence in such a way that this information is lost, i.e. the shifted sequence is
an i.i.d.\ sequence of coins with the original bias. This means that we wish to embed the invariant distribution~$\nu=m$ given
by \smash{$m_{\texttt{head}}=p, m_{\texttt{tail}}=1-p$}. Theorem~\ref{T:characterization_thm_intro}
shows that this is impossible.
\end{example}
\pagebreak[3]

\begin{example}(Extra head problem with a finite pattern)\label{Ex:extra_head_with_pattern}
In the setup of Example~\ref{Ex:extra_head} we now ask to find a particular \emph{finite pattern} of successive outcomes, such that the coins to its
left and right remain an i.i.d.\ sequence of coins with the same bias. Looking, for example, for the pattern  $\texttt{head}/\texttt{tail}$ we would first
reveal the coin at the origin, and then if this shows $\texttt{head}$ its right neighbour, and if this shows $\texttt{tail}$ its left neighbour. The 
underlying Markov chain has the state space $\{\texttt{tail}/\texttt{tail}, \texttt{tail}/\texttt{head}, \texttt{head}/\texttt{tail}, \texttt{head}/\texttt{head}\}$, 
the transition matrix
$$\left( \begin{array}{cccc} 1-p & p & 0 & 0\\ 0 & 0& 1-p & p\\ 1-p & p & 0& 0\\ 0& 0 & 1-p & p \\ \end{array} \right),$$
and invariant measure $((1-p)^2,\, p(1-p),\, p(1-p),\, p^2)$. Our theorem shows that, if we initially reveal $\texttt{tail}/\texttt{tail}$
then we need $1/p$ to be an integer, and if we reveal $\texttt{head}/\texttt{head}$ then we need $1/(1-p)$ to be an integer. 
Hence we can only embed $\texttt{head}/\texttt{head}$ if $p=\frac12$. More generally, the problem
can be solved for patterns that are repetitions of the single symbol $\texttt{head}$ if and only if $1/p$ is an integer,
for patterns that are repetitions of the single symbol $\texttt{tail}$ if and only if $1/(1-p)$ is an integer,
and for patterns containing both symbols $\texttt{tail}$ and $\texttt{head}$ if and only if $p=\frac12$.
\end{example}
\ \\[-7mm]

\begin{example}(Simple random walk)\label{Ex:SRW}
Let $X$ be a two-sided simple symmetric random walk on the integers, with $X_{0}=i$ for some $i \in \Z$. In this case
the invariant measure is $m_{i}=1$ for all $i \in \Z$, hence Theorem~\ref{T:characterization_thm_intro} shows that the target
measures that can be embedded are precisely the Dirac measures $\delta_{j}$, $j \in \s$. The same result holds for the simple
symmetric random walk on the square lattice~$\Z^2$.
\end{example}
\ \\[-7mm]

The proof of Theorem~\ref{T:characterization_thm_intro} extends the ideas developed by Liggett~\cite{LT02}
and Holroyd and Peres~\cite{HP05} for the extra head problem to the more general Markov chain setup. In particular, under the
additional assumption that the target measure does not charge the initial state, we are able to generalize Liggett's construction of
an elegant explicit solution, in analogy to the Brownian motion case studied in Last et al.~\cite{LMT14}. Recall that the  case when the target
measure charges the initial state was already discussed in
Example~\ref{Ex:characterization_thm_intro}. To describe this solution we define the local time~$L^j$ spent by $X$ at
state $j \in \mathcal{S}$  to be the normalized counting measure given by
$$L^{j}(A):= \frac1{m_j}\,\#\{ n\in A \colon X_{n}= j\} \qquad \mbox{ for any $A\subset \Z$.}$$

\begin{thm}\label{T:Existence_allocation_rule_intro}
Let $X$ be a two-sided irreducible and recurrent Markov chain with $X_{0}=i$ and further assume that the target measure $\nu$ satisfies $\nu_i=0$ and
the conditions in Theorem~\ref{T:characterization_thm_intro}$\,(b)$. Then
\begin{align}
T_*:&=\min\Big\{n\ge 0 \colon L^{i}([0,n]) \le \sum_{j \in \s} \nu_j\,L^{j}([0,n])\Big\} \label{D:tau}
\end{align}
is a finite, non-randomized stopping time satisfying the conditions of Theorem~\ref{T:characterization_thm_intro}$\,(a)$.
\end{thm}

\begin{example}
We take a stationary three state Markov chain with transition probabilities given by \smash{$p_{12}=p_{32}=1$} and
\smash{$p_{21}=1-p$} and \smash{$p_{23}=p$}. If $1/p$ is an integer we can shift the chain so that it starts in the
third state and the chain property is preserved, as follows: Uncover the state at the origin.
If it is the third state we are done; if it is the second state we  move along the chain until the number of visits to the third state is at least {$p$} times  the number
of visits to the second state; if it is the first state we move until the number of visits to the third state is at least
\smash{$\frac{p}{1-p}$} times the number of visits to the first state. Note that if the state of the origin is the first state it is \emph{not} a solution
to wait one time step, whence you are in the second state, and then apply the strategy for start in the second state as this creates a bias in the backward chain.
\end{example}\pagebreak[3]

Skorokhod embedding problems usually concern embedding times with finite expectation. However in the extra head problem it is not possible to achieve finite expectation
of the random time~$T$. In fact Liggett~\cite{LT02} shows that in this case always $E\sqrt{T}=\infty$, see also~Holroyd and Liggett~\cite{HL01}. For the simple random walk on
the integers we expect in analogy to the Brownian motion case studied by Last et al.~\cite{LMT14} that always $E\sqrt[4]{T}=\infty$. Our aim here is to understand the general picture.
\smallskip

To this end  we now recall
the notion of asymptotic Green's function of the Markov chain. Given states $i,j\in\mathcal S$ we first define the normalized truncated Green's function by
$$a_{ij}(n)= \E_{i} L^j([0,n]) =  \frac{1}{m_{j}}\, \E_{i}\Big[ \sum_{k=0}^{n} \mathbbm{1} \{ X_{k} = j\}\Big],$$
that is $a_{ij}(n)$ gives the normalized expected number of visits to state $j$ between time~$0$ and time $n$,
by the Markov chain with initial state $X_{0}=i$. 
By Orey's ergodic theorem, see, e.g., Chen~\cite{XC99},
for any states $i,j,k,l\in\mathcal S$, the functions
$a_{ij}$ and $a_{kl}$ are asymptotically equivalent in the sense that
$$\lim_{n\to\infty} \frac{a_{ij}(n)}{a_{kl}(n)}=1.$$
We then define the \emph{asymptotic Green's function} $a(n)$ as the equivalence
class of the truncated Green's functions under asymptotic equivalence. Observe that
finiteness of moments is a class property, i.e.~expressions of the form $E[a(Y)]<\infty$,
where $a$ is an equivalence class and $Y$ an integer-valued random variable, are meaningful.
\medskip
\pagebreak[3]

\begin{thm}\label{T:moment_prop}
Let $X$ be a two-sided irreducible and recurrent Markov chain with $X_{0}=i$ and $\nu$ be any target measure
different from the Dirac measure~$\delta_i$.
If $T_*$ is the stopping time defined in~\eqref{D:tau},~then
\begin{enumerate}
\item[(i)]
$\E_{i}\big[a(T_*)^{1/2}\big]=\infty.$
\end{enumerate}
If additionally $\nu$ has finite support, then
\begin{enumerate}
\item[(ii)]
$\E_{i}\big[a(T_*)^{\beta}\big] < \infty$ for all $0\leq \beta < \frac12$.
\end{enumerate}
\end{thm}
\smallskip

As $a(n)$ cannot grow faster than $n$, our solutions~$T_*$ always have `bad' moment properties
as even for the nicest Markov chain $T_*$ can never have finite square root moments. 
However, our next theorem shows that no other solution of the embedding problem has better moment properties than~$T_*$.
\smallskip

In fact, it turns out that~$T_*$ has a strong optimality property, as it simultaneously minimizes all concave moments of non-negative solutions of
the embedding problem. This striking result is new even for the case of the extra head problem and therefore, in our opinion,
constitutes the most interesting contribution in this paper.%
\smallskip%

\begin{thm}\label{T:Optimality}
Let $X$ be a two-sided irreducible and recurrent Markov chain with $X_{0}=i$ and $\nu$ be a target measure
satisfying the conditions in Theorem~\ref{T:Existence_allocation_rule_intro}.  If $T_{*}$ is the solution of the Skorokhod embedding
problem constructed in~\eqref{D:tau} and $T$ any other non-negative (possibly randomized) solution,
then
$$\phantom{X^{X^{X^X}}}
\E_{i}\big[ \psi(T_{*}) \big]  \le \E_{i}^{_{^\oplus}}\big[ \psi(T) \big], \\[1mm]
$$
for any non-negative concave function~$\psi$ defined on the non-negative integers,
where the expectation on the right is with respect to the chain as well as any possible extra randomness used to define~$T$.
\end{thm}

\pagebreak[3]

Theorem~\ref{T:Optimality} is inspired by exciting recent developments connecting the classical Skorokhod embeddings for Brownian motion
with optimal transport problems. In a recent paper, Beiglb\"ock, Cox and Huesmann~\cite{BH13} exploit this connection to characterize certain
solutions to the Skorokhod embedding problem by a geometric property. In a similar vein, our solution $T_*$ is characterized by a geometric
property, the `non-crossing' condition, which yields the optimality. See also our concluding remarks in Section~\ref{S:Conclusion} for possible extensions of this result.
\medskip

\begin{example}
Suppose the underlying Markov chain is \emph{positive recurrent}. Then the asymptotic Green's function satisfies $a(n)\sim n$.
Therefore all non-negative solutions $T$ of the Skorokhod embedding problem satisfy \smash{$\E_{i}[\sqrt{T}]=\infty$}, while the solution
constructed in Theorem~\ref{T:Existence_allocation_rule_intro} satisfy \smash{$\E_{i}[T_*^{\beta}] < \infty$} for all $0\leq\beta<1/2$. This applies
in particular to Examples~\ref{Ex:extra_head} and~\ref{Ex:extra_head_with_pattern}.
\end{example}

\pagebreak[3]

\begin{example}
The situation is much more diverse for \emph{null-recurrent} chains. Looking at Example~\ref{Ex:SRW}, for a two-sided
simple symmetric random walk on the integers we have \smash{$a(n) \sim \sqrt{n}$}. Hence the solution \smash{$T_*$}  constructed in
Theorem~\ref{T:Existence_allocation_rule_intro} satisfies \smash{$\E_{i}[T_*^{\alpha}] < \infty$} for all $0\leq\alpha<1/4$,
while any non-negative solution  has infinite $1/4$ moment. This is similar to the case of Brownian motion on the line, which is discussed
in~\cite{LMT14}, although in that paper other than here the discussion is restricted to solutions which are non-randomized stopping times.
In contrast to this, for simple symmetric random walk on the square lattice~\smash{$\Z^{2}$} we have $a(n) \sim \log{n}$, and therefore
\smash{$\E_{i}[\sqrt{\log{T}}]$} is infinite for any non-negative solution~$T$, while the solution~\smash{$T_*$}  constructed in
Theorem~\ref{T:Existence_allocation_rule_intro} satisfies \smash{$\E_{i}[(\log{T_*})^{\alpha}]<\infty$}, for all $0\leq\alpha<1/2$.
\end{example}
\pagebreak[3]

\section{Relating embedding and allocation problems}\label{S:Characterization}

In this section we relate our embedding problem to an equivalent allocation problem. The section specializes 
some results from Last and Thorisson~\cite{LT09} which are themselves  based on ideas from~\cite{HP05}. We give complete proofs of the known facts in order  to keep this paper self-contained. Generalizing from~\cite{LMT14} we call a random time~$T$ an \emph{unbiased shift} of
the Markov chain $X$ if the shifted process $T^{-1}X$ is a two-sided Markov chain
with the same transition matrix as~$X$. Note that this definition allows $T$ to be randomized, i.e. it does not have
to be a function of the sample chain~$X$ alone.
\smallskip

Let $\Omega= \{(\omega_{i})_{i \in \Z} \colon \omega_{i} \in \s \}$ be the set of trajectories of $X$. A \emph{transport rule} is a
measurable function  $\theta \colon\Omega \times \Z \times \Z \to [0,1]$ satisfying
$$\sum_{y \in \Z}\theta_\omega(x,y)=1 \qquad \mbox{ for all $x \in \Z$ and $\p$-almost every $\omega$.}$$
Note that we write the dependence on the trajectory $\omega$ by a subindex, which we drop from the notation whenever convenient.
Transport rules are interpreted as distributing mass from $x$ to $\Z$ in such a way that the site $y$ gets a proportion $\theta(x,y)$
of the mass. For sets $A,B \subset\Z$ we define
\[
\theta_{\omega}(A,B) := \sum_{x \in A, y \in B}\theta_{\omega}(x,y).
\]
A transport rule $\theta$ is called \emph{translation invariant} if
$$\theta_{z\omega}(x+z,y+z)=\theta_{\omega}(x,y),$$
for all $\omega \in \Omega$ and $x, y, z \in \Z$, where $z\omega$, defined by \smash{$z\omega_n=\omega_{n-z}$} for any $n \in \Z$,
is the trajectory shifted by $-z$.  A transport rule \emph{balances} the random measures $\xi$ and $\zeta$ on~$\Z$ if
\begin{equation}\label{E:balanced_transp_prop}
\sum_{z \in \Z}\theta_{\omega}(z,A)\xi(z)=\zeta(A),
\end{equation}
for any $A \subset \Z$ and $\p$-almost all $\omega$.
%
Given a two-sided Markov chain $X$ as before recall the definition of the local times $L^{i}$, and
given a probability measure $\nu=(\nu_i \colon i\in\s)$ we further define
\[
L^{\nu}=\sum_{i \in \s}\nu_i \,L^{i}.
\]

\pagebreak[3]

\begin{prop}\label{P:transport_rule}
Assume that there is a measurable family of probability measures $({\mathbb Q}_\omega \colon \omega\in\Omega)$ on some measurable space $\Omega'$
and $T\colon \Omega \times \Omega' \to \Z$ is measurable. The random time $T$ and a translation invariant  transport rule $\theta$ are associated if
\begin{equation}\label{E:P_Theta}
\mathbb{Q}_\omega\big(\omega'\in\Omega' \colon T(\omega, \omega')=t\big)=\theta_{\omega}(0,t) \qquad \mbox{ for all } t\in \Z \mbox{ and $\mathbb P$-almost all } \omega\in\Omega.
\end{equation}
For any probability measure $\mu=(\mu_i \colon i\in \s)$ we define the probability measure~$\p_\mu^{_{^\oplus}}$ on $\Omega\times\Omega'$ by
\begin{equation}\label{E:P_mu}
\prob_\mu^{_{^\oplus}}(d\omega \, d\omega')=\sum_{i\in\s} \mu_i \,\prob_i(d\omega) \, {\mathbb Q}_\omega (d\omega').
\end{equation}
Then, if $\mu, \nu$ is any pair of
probability measures on $\s$ and the random time $T$ and translation invariant transport rule $\theta$ are associated, the following statements are equivalent.
\begin{itemize}
\item[$(a)$] Under $\prob_\mu^{_{^\oplus}}$ the random time  $T$ is an unbiased shift of $X$ and $X_T$ has law $\nu$.
\item[$(b)$] The transport rule $\theta$  balances $L^\mu$ and $L^\nu$ $\p$-almost everywhere.
\end{itemize}
\end{prop}

Note that in the last proposition unbiased shifts need not be non-randomized. The transport rules associated to non-randomized shifts are the  \emph{allocation rules}. These are given by a measurable map $\tau\colon \Omega \times \Z \to \Z$
such that $\theta_\omega(x,y)=1$ if $\tau_\omega(x)=y$ and zero otherwise.

\begin{prop}\label{P:allocation_rule}
If the random time $T$ in Proposition~\ref{P:transport_rule} is non-randomized, then there is an associated transport rule $\theta$,
which is an allocation rule. Conversely if $\theta$ in Proposition~\ref{P:transport_rule} is an allocation rule, then there exists an associated
non-randomized random time~$T$.
\end{prop}

We give proofs of the propositions for completeness. For a transport rule $\theta$ we define
\begin{equation}\label{E:def_J}
J_{\mu}(\omega):=\sum_{k \in \Z} \theta_{\omega}(k,0)\, L^{\mu}(k),
\end{equation}
which is interpreted as the total mass received by the origin. We recall the following simple fact,  see~\cite{HP05} for
a more general version.

\begin{lemma}\label{L:mass_transport_pr}
Let $m\colon \Z \times \Z \to [0, \infty ]$ be such that $m(x+z,y+z)=m(x,y)$ for all $x,y,z \in \Z$. Then
\[
\sum_{y \in \Z}m(x,y) = \sum_{y \in \Z}m(y,x).
\]
\end{lemma}

The following calculation is at the core of the proof.

\begin{lemma}\label{L:product}
Suppose that $T$ and $\theta$ are related by~\eqref{E:P_Theta}. Then, for any measurable function
$f \colon \Omega \to [0,\infty]$, we have
\[
\E_{\mu}^{_{^\oplus}}\big[f(T^{-1}X)\big]= \E \big[J_{\mu}(X)f(X)\big],\\[2mm]
\]
where $\E_{\mu}^{_{^\oplus}}$ is the expectation with respect to $\prob_\mu^{_{^\oplus}}$ defined in \eqref{E:P_mu}.
\end{lemma}

\begin{proof}[Proof of Lemma~\ref{L:product}]
Writing $\p_\mu=\sum_{i\in\Z} \mu_i \p_i$ we get
\begin{align*}
\E_{\mu}^{_{^\oplus}}\big[f(T^{-1}X)\big]&=\int d\p_{\mu}(\omega) \int  f\big(T(\omega, \omega')^{-1}X(\omega)\big) \, {\mathbb Q}_\omega(d\omega') \\
&=\int d\p_{\mu}(\omega) \sum_{t \in \Z} \, {\mathbb Q}_\omega(T=t) f(t^{-1}X(\omega)).
\end{align*}
Using relation~\eqref{E:P_Theta} and the definition of $\p_\mu$ we continue with
\begin{align*}
\phantom{\E_{\mu}^{_{^\oplus}}\big[f(T^{-1}X)\big] xxx}&=\sum_{i\in\Z}\mu_{i}\int d\p_{i}(\omega) \sum_{t \in \Z} \theta_{\omega}(0,t)f(t^{-1}X(\omega)) \\
&= \sum_{i\in\Z}\mu_{i}\int d\p(\omega) \sum_{t \in \Z} \theta_{\omega}(0,t)L^{i}(0)f(t^{-1}X(\omega)),
\end{align*}
as $L^{i}(0)=\frac1{m_i}$ and $L^{j}(0)=0$ under $\p_i$ for $j\not=i$. Applying Lemma~\ref{L:mass_transport_pr}
gives
\begin{align*}
\phantom{\E_{\mu}\big[f(T^{-1}X)\big] xxx}
&=\sum_{i\in\Z} \mu_{i} \int d\p(\omega) \sum_{t \in \Z} \theta_{\omega}(t,0)L^{i}(t)f(X(\omega)) \phantom{xx}\\
& =  \int d\p(\omega) \sum_{t \in \Z} \theta_{\omega}(t,0) L^{\mu}(t)f(X(\omega))\\
&=\E\big[J_{\mu}(X)f(X)\big],
\end{align*}
using first the definition of $L^\mu$ and second the definition of $J_{\mu}(X)$.
\end{proof}

\begin{proof}[Proof of Proposition~\ref{P:transport_rule}]
First assume that $\theta$ is a translation invariant transport rule. Then, for any non-negative measurable $f$, by
Lemma~\ref{L:product}, we have
\begin{align}
\E_{\mu}^{_{^\oplus}}\big[f(T^{-1}X)\big] =\E\big[J_{\mu}(X)f(X)\big]=
\E\big[ \sum_{k \in \Z} \theta_{\omega}(k,0)\, L^{\mu}(k) f(X)\big]. \label{bas}
\end{align}
If $\theta$ balances $L^\mu$ and $L^\nu$ this equals
\begin{align*}
\E\big[ L^{\nu}(0) f(X) \big]
=\sum_{j \in\Z}\nu_{j}\, \E\big[L^{j}(0)f(X)\big]
=\sum_{j \in \Z} \nu_{j}\, \E_{j}\big[f(X)\big]=\E_{\nu}\big[ f(X)\big].
\end{align*}
Hence under $\p_\mu^{_{^\oplus}}$ the random variable $T^{-1}X$ has the law of $X$ under~$\prob_\nu$.
In other words $T$ is an unbiased shift and $X_T$ has distribution~$\nu$.

Conversely, assume that $T$ is an unbiased shift and $X_T$ has distribution~$\nu$. Hence
$\E_{\mu}^{_{^\oplus}}[f(T^{-1}X)] =\E_{\nu}[f(X)]= \E[ L^{\nu}(0) f(X)].$
Plugging this into~\eqref{bas} gives
$$\E\big[ \sum_{k \in \Z} \theta_{\omega}(k,0)\, L^{\mu}(k) f(X)\big]=\E\big[ L^{\nu}(0) f(X) \big].$$
As $f$ was arbitrary we get
$\sum_{k \in \Z} \theta_{\omega}(k,0)\, L^{\mu}_\omega(k)=L^{\nu}_\omega(0)$ for $\p$-almost all~$\omega$,
where we emphasise the dependence of the measures $L^\mu$ and $L^\nu$ on the trajectories by a subscript. As
$\theta$ is translation invariant we get, substituting $m:=k-\ell$,
\begin{align*}
\sum_{k \in \Z}\theta_{\omega}(k,A)\, L^\mu_\omega(k) & =
\sum_{k \in \Z}\sum_{\ell \in A} \theta_{\omega}(k,\ell)\, L^{\mu}_\omega(k)
= \sum_{\ell \in A} \sum_{m \in \Z} \theta_{-\ell\omega}(m,0)\, L^{\mu}_{-\ell\omega}(m)\\
& = \sum_{\ell \in A} L^{\nu}_{-\ell\omega}(0)= \sum_{\ell \in A} L^{\nu}_{\omega}(\ell) = L^{\nu}_{\omega}(A),
\end{align*}
for every $A\subset\Z$ and $\p$-almost every~$\omega$.
\end{proof}
\pagebreak[3]

\begin{proof}[Proof of Proposition~\ref{P:allocation_rule}]
Suppose $T=T(\omega)$ is non-randomized. Define $\tau_\omega\colon \Z\to\Z$ by $\tau_\omega(k)=T(-k\omega)+k$
and let $\theta_\omega(x,y)=1$ if $\tau_\omega(x)=y$ and zero otherwise. Then $\theta$ is a translation invariant
allocation rule. Moreover, $\Q_\omega(T=t)=\mathbbm{1}\{t=T(\omega)\}=\mathbbm{1}\{t=\tau_\omega(0)\}=\theta_\omega(0,t)$,
hence~$T$ and $\theta$ are associated.
Conversely, if $\theta$ is a translation invariant allocation rule given by $\tau\colon\Omega\times\Z \to \Z$ define
a non-randomized time $T$ by $T=\tau_\omega(0)$. As before, $\Q_\omega(T=t)=\mathbbm{1}\{t=T(\omega)\}=\mathbbm{1}\{t=\tau_\omega(0)\}
=\theta_\omega(0,t)$, and hence~$T$ and $\theta$ are associated.
\end{proof}

\section{Existence of allocation rules: Proof of Theorems~1 and~2}\label{S:Existence}

In the light of the previous section our Theorems~\ref{T:characterization_thm_intro} and~\ref{T:Existence_allocation_rule_intro}
can be formulated and proved as equivalent statements about allocation rules. We start with the result on non-existence of
non-randomized unbiased shifts, which is implicit in Theorem~\ref{T:characterization_thm_intro}.
\medskip

Suppose that statement~$(a)$ in Theorem~\ref{T:characterization_thm_intro} holds and for the Markov chain $X$ with $X_0=i$ there exists a
non-randomized unbiased shift $T$ such that $X_{T}$ has law $\nu$.
Then by Proposition~\ref{P:allocation_rule} there exists a translation-invariant allocation rule $\tau$ associated with~$T$ and by Proposition~\ref{P:transport_rule}
this rule balances the measures $L^i$ and $L^\nu$. Recall that $L^i$ is the measure on $\Z$ which has masses of fixed size $1/m_i$ at the times
when the stationary chain~$X$ visits state~$i$. By the balancing property~\eqref{E:balanced_transp_prop} for allocation rules, all masses of $L^\nu$ must have sizes which are integer
multiples of $1/m_i$. As these masses are $\nu_j/m_j$ we get that \smash{$\frac{m_i}{m_j} \nu_j$} must be integers for all $j\in\s$, which is statement~$(b)$.
\medskip

The remainder of this section is devoted to the proof of \emph{existence} of non-randomized unbiased shifts of the Markov chain $X$ with $X_0=i$,
embedding~$\nu$ under the assumption of
Theorem~\ref{T:characterization_thm_intro}\,$(b)$. By Example~\ref{Ex:characterization_thm_intro} we may additionally assume that for the initial state~$i$ of the Markov chain
we have $\nu_i=0$. Our claim is that the stopping time $T_*$ defined in Theorem~\ref{T:Existence_allocation_rule_intro} is an unbiased shift with the required properties. 
The next proposition shows
that an associated allocation rule balances the measures $L^i$ and $L^\nu$ which, once accomplished, implies Theorem~\ref{T:Existence_allocation_rule_intro}
and completes the proof of Theorem~\ref{T:characterization_thm_intro}.
\medskip

\begin{prop}\label{T:existence_allocation_rule_extended}
Under the assumptions set out above, the following holds.
\begin{itemize}
\item[$(a)$] The mapping~$\tau\colon \Omega \times \Z\to\Z$ defined by
$$\tau_\omega(k)=\min\big\{n\ge k \colon L_\omega^{i}([k,n]) \le L_\omega^{\nu}([k,n])\big\}$$
is a translation-invariant allocation rule associated with the $T_{\ast}$ defined in~\eqref{D:tau}.\\[1mm]

\item[$(b)$] For $\p$-almost every~$\omega$ and all $A \subset \Z$ we have
\begin{equation}\label{b}
\sum_{k \in \Z}\mathbbm{1} \{\tau_\omega(k) \in A \}\, L_\omega^{i}(k)=L^{\nu}_\omega(A),
\end{equation}
in other words the allocation rule balances $L^i$ and~$L^\nu$.
\end{itemize}
\end{prop}

\medskip

The proof of the proposition is similar to that of \cite[Theorem 5.1]{LMT14} in the diffuse case.
We prepare it with two lemmas. The first lemma is a pathwise statement which holds
for every fixed trajectory~$\omega$ satisfying the stated assumption.

\begin{lemma}\label{L:allocation_rule_lemma}
Suppose $b \in \Z$ is such that $X_{b}=j$ for some $j \in \s$ with $\nu_j>0$, and $a\in\Z$ is given by
$$a:=\max\big\{k < b \colon  L^{i}([k,b]) \geq L^{\nu}([k,b])\}.$$
Then 
\begin{equation}\label{E:allocation_rule_lemma}
\sum_{k \in [a,b]}\mathbbm{1}\{\tau(k) \in A \} \, L^{i}(k)= L^{\nu}(A),
\end{equation}
holds for any $A \subset [a,b]$.
\end{lemma}

\begin{proof} We define the function $\Delta f\colon\Z\to[0,\infty)$ by
\[
\Delta f(k):=L^{i}(k)-L^{\nu}(k)=\left\{
\begin{array}{l l}
    \frac1{m_i} & \quad \text{if} \quad X_{k} = i, \\[2mm]
    -\sfrac{\nu_j}{m_{j}} & \quad \text{if} \quad X_{k} = j \not=i.
  \end{array} \right.
\]
Recall that by our assumption $\sfrac{\nu_j}{m_{j}}$ is an integer multiple of $\sfrac1{m_i}$. Hence, denoting
\[
f_{u}^{v}:=\sum_{n=u}^{v}\Delta f(n) \quad \mbox{ for all $u,v \in \Z$ and $u \le v$,}
\]
we have $a=\max\big\{k < b \colon  f_k^b = 0\}$ and hence $f_a^b=0$.

By the additivity of both sides of~\eqref{E:allocation_rule_lemma} it suffices to prove
\begin{equation}\label{auxi}
\sum_{k \in [a,b]}\mathbbm{1}\{\tau(k) =z \} \, L^{i}(k)= L^{\nu}(z) \quad \mbox{for all sites $z\in[a,b]$.}
\end{equation}
Fix $z\in[a,b]$ and let $j=X_z$. Observe that $\tau(k)=z$ if and only if $f_k^z\leq0$
but $f_k^\ell>0$ for all $k\leq\ell<z$. Hence we may assume $\nu_j>0$ as otherwise both sides
of~\eqref{auxi} are zero.  We also have that $f_a^z>0$ if $z<b$. Indeed, suppose that
\smash{$f_a^z\leq0$}.  Then \smash{$f_{z+1}^b=f_a^b-f_a^z\geq 0$}  contradicting the choice of~$a$.

As \smash{$f_a^z\geq0, f_z^z=-\frac{\nu_j}{m_j}<0$} and \smash{${\nu_j}/{m_j}$} is an integer multiple of
\smash{$1/{m_i}$} we find a $k_1\geq a$ with $f_{k_1}^z=0$ and $f_j^z<0$ for all $k_1<j\leq z$.
Similarly, we find $k_1< k_2<\cdots<k_N$  where \smash{$N:=(\sfrac{m_i}{m_j})\,\nu_j$} such that
$$f_{k_n}^z=\sfrac{1-n}{m_i} \mbox{  and $f_j^z<\sfrac{1-n}{m_i}$ for all $k_n<j\leq z$.}$$
As $\tau(k)=\min\{n\ge k \colon f_k^n\leq 0\}$ we infer that $\tau(k_n)=z$ for all $n\in\{1,\ldots, N\}$
and there are no other values $k$ with $\tau(k)=z$. Each of these values contributes a summand
\smash{$\sfrac1{m_i}$} to the left hand side in~\eqref{auxi}. Therefore this side equals \smash{$\sfrac{N}{m_i}=
\sfrac{\nu_j}{m_j}$}, as does the right hand side. This completes the proof.
\end{proof}

The second lemma is probabilistic and ensures in particular that the mapping $\tau$
described in Proposition~\ref{T:existence_allocation_rule_extended}\,$(a)$ is well defined.

\begin{lemma}\label{prob}
For $\p$-almost every $\omega$ the following two events hold
\begin{itemize}
\item[(\emph{E1})] for all $k$ with $X_{k}= i$ we have $\tau(k) < \infty$;\\[-2mm]
\item[(\emph{E2})] for all $b$ such that $X_{b}=j$ for some $j \in \s$ with $\nu_j>0$
there exists $a < b$  such that $X_{a}=i$ and $L^{i}([a, b])= L^{\nu}([a,b])$.
\end{itemize}
\end{lemma}

\begin{proof}
To show this we use an argument from~\cite{HP05}, see Theorem 17 and the following remark. We formulate the negation
of the two events. The complement of (\emph{E1}) is the event that there exists~$k$ such that $X_{k}=i$ and
$L^{i}([k, \ell]) > L^{\nu}([k,\ell])$, for all $\ell>k$. The complement of (\emph{E2}) is that
there exists $b$ such that $X_b=j$ for some $j\in\s$ with $\nu_j>0$ and $L^{i}([a,b]) < L^{\nu}([a,b])$,
for all $a<b$ with $X_a=i$. We first show that, for $\p$-almost every~$\omega$, both complements cannot occur simultaneously.

Indeed, for a fixed $\omega$, it is clear that there cannot be $k$ and $b$ as above such that $k<b$. Assume for
contradiction that the set of trajectories $\omega$ for which there exist $k>b$ as above has positive probability.
On this event the minimum over all $k$ with $\tau(k)=\infty$ for all
$\ell>k$ is finite, we denote it by~$K$. By translation invariance $\p(K=0)>0$
from which we infer by conditioning on the event $\{X_0=i\}$ that $\p_i(K=0)>0$. If $(T_n\colon n\in\N)$
is the collection of return times to state~$i$, by the invariance described in Example~1.1 we have
$\p_i(K=T_n)=\p_i(K=0)>0$ for all $n\in\N$ contradicting the finiteness of~$\p_i$. Therefore we have shown that, for
$\p$-almost every~$\omega$,  either~(\emph{E1}) or~(\emph{E2}) occurs.

As the last step we show that event (\emph{E1}) cannot occur without event (\emph{E2}). To this end define
$m(x, y)= \E[\mathbbm{1}\{\tau(x)=y, X_x=i\}]$ and apply Lemma~\ref{L:mass_transport_pr} to get
$$\E\Big[\sum_{k\in \Z}\mathbbm{1}\{\tau(k)=0, X_k=i\}\Big]= \E\Big[\sum_{k\in \Z}\mathbbm{1}\{\tau(0)=k, X_0=i\}\Big].$$
%
The left-hand side in this equation equals~$m_i$ if and only if~(\emph{E2}) occurs
$\p$-almost every $\omega$, and the right-hand side equals~$m_i$ if and only if~(\emph{E1}) occurs
$\p$-almost every $\omega$. As these two events cannot fail at the same time, both events (\emph{E1}) and
(\emph{E2}) occur for $\p$-almost every $\omega$.
\end{proof}

\begin{proof}[Proof of Proposition~\ref{T:existence_allocation_rule_extended}]
Recall that $\tau$ is well-defined and note that translation-invariance of the allocation rule defined
in terms of~$\tau$ follows easily from the fact that $\tau_\omega(k)=\tau_{k\omega}(0)+k$. As $T_*(\omega)=
\tau_\omega(0)$ by definition, the allocation rule  is associated with $T_*$. This proves~$(a)$.

To prove~$(b)$ we note that it suffices to fix~$z\in\Z$ and show that for $\p$-almost every~$\omega$ equation~\eqref{b}
holds for $A=\{z\}$. We let $b=\tau(z)$. By Lemma~\ref{prob}  for $\p$-almost every~$\omega$ there exists $a<b$ such that
$X_{a}=i$ and $L^{i}([a, b])= L^{\nu}([a,b])$. Then the interval $[a,b]$ contains~$z$ and all $k$ with $\tau(k)=z$.
Hence the results follows by application of Lemma~\ref{L:allocation_rule_lemma}.
\end{proof}

\section{Moment properties of~$T_*$: Proof of Theorem~3}\label{S:Moment}

The critical exponent~$\frac12$ occurring in Theorem~\ref{T:moment_prop} originates from the behaviour of the first passage
time below zero by a mean zero random walk. We summarize the results required for such random walks
in the following lemma.

\begin{lemma}\label{firstpassage}
Let $\xi, \xi_1, \xi_2, \ldots$ be independent identically distributed random variables with $E \xi=0$
taking values in the integers. Define the associated random walk by $S_n=\sum_{i=1}^n \xi_i$ and its first
passage time below zero as $N=\min\{n\in\N \colon S_n\leq 0\}$.
\begin{itemize}
\item[$(a)$] If the walk is skip-free to the right, i.e.\ ${P}(\xi>1)=0$, then
$\displaystyle E\big[ N^{1/2}\big]=\infty.$
\item[$(b)$] If the walk has finite variance, then there exists $C>0$ such that
$\displaystyle P\big(N>n\big)\sim C\, \frac1{\sqrt{n}}.$
\end{itemize}
\end{lemma}

\begin{proof}
$(a)$ Denote by $N^{\ssup j}$ the first passage time for the walk given by $S^{\ssup j}_n=\sum_{i=1}^n \xi_{i+j-1}$. Then
$$E[N \wedge n ] =  \sum_{j=1}^n {P}(N\geq j) =
\sum_{j=1}^n {P}(N^{\ssup j} \geq n-j+1)
= E \Big[ \sum_{j=1}^n \1\{N^{\ssup j} \geq n-j+1\}\Big],$$
If $\underline{S}_n$ denotes the minimum of $\{S_0, S_1, \ldots, S_n\}$ we have, using that the walk is skip-free to the right,
$$\sum_{j=1}^n \1\{N^{\ssup j}\geq n-j+1\} = S_n - \underline{S}_n.$$
This implies $E[N \wedge n ] \geq E[(S_n)_+].$ By a concentration inequality for arbitrary sums of
independent random variables,
see \cite[Theorem~2.22]{P95},
there exists a constant $C>0$ such that, for all $\eps>0$ and~$n\in\N$, we have
${P}( S_n \in [-\eps \sqrt{n}, \eps \sqrt{n}]) \leq C\, \eps.$
Hence, by  Markov's inequality, for any $\eps>0$,
\begin{align*}
E \big[(S_n)_+\big] & = \frac12 E|S_n|
\geq \frac12 \,\eps\sqrt{n} \, {P}\big( |S_n| > \eps \sqrt{n} \big)
\geq \frac12\, \eps(1-C\eps)\, \sqrt{n}.
\end{align*}
We infer that \smash{$\liminf  \frac1{\sqrt{n}} \,E[N \wedge n ] >0$}.
But if we we had $E[ N^{1/2}]<\infty$ dominated convergence would imply
that this limit is zero, which is a contradiction.

$(b)$ This is a classical result of Spitzer~\cite{Spitzer60}. A good proof can be found in \cite[Theorem~1a in Section~XII.7]{Feller},
see also \cite[Section XVIII.5]{Feller} for a proof that random walks with finite variance satisfy Spitzer's condition.
\end{proof}

\subsection{Proof of Theorem~\ref{T:moment_prop}\,$(i)$.}
We start by proving a variant of the upper half in the Barlow-Yor inequality~\cite{BY81}
for Markov chains. This result, usually given in the context of continuous martingales,
estimates the moments of the local time at a stopping time, by moments
of the stopping time itself.%

\begin{lemma}\label{L:local_time_inequality}
For any $0<p<\infty$, there  exists a constant $C_{p}$ such that,
for any state~$i\in\s$ and any stopping time~$T$,
\begin{equation}\label{E:local_time_inequality}
\E_{i}[L^{i}([0,T])^{p}] \le C_{p}\, \E_{i}[a_{ii}(T)^{p}].
\end{equation}
\end{lemma}

The lemma relies on the following classical inequality, we refer to~\cite[(6.9)]{Bass95} for a proof.

\begin{lemma}[\textbf{Good $\lambda$ inequality}]\label{L:good_lambda}
For every  $0<p<1$ there is a constant $C_{p}>0$ such that, for any pair of non-negative
random variables $(X,Y)$ satisfying 
\begin{equation}\label{E:good_lambda_ineq}
P(X > 3 \lambda, Y <\delta \lambda) \le \delta\,P(X> \lambda)
\quad \mbox{for all $0<\delta<3^{-p-1}$ and $\lambda > 0$,}
\end{equation}
we have
$$E\left[ X^{p}\right] \le C_{p} \, E\left[Y^{p}\right].$$
\end{lemma}

\begin{proof}[Proof of Lemma~\ref{L:local_time_inequality}] If we show that \eqref{E:good_lambda_ineq} holds with random
variables $X=m_i L^{i}([0,T])$ and $Y=m_i a_{ii}(T)$ under $\p_i$,  the result follows immediately from Lemma~\ref{L:good_lambda}.
If $\lambda\leq1$ the left hand side of~\eqref{E:good_lambda_ineq} is zero and there is nothing to show. We may therefore assume
that $\lambda>1$. Define $m_ia_{ii}^{-1}(x):=\max\{ n \colon m_ia_{ii}(n)<x\}$. Let $T_0=0$ and $T_k$ be the time of the $k$th
visit of state~$i$ after time zero. Finally assume, without loss of generality, that $\p_{i}(X>\lambda)>0$. Then,
\begin{align*}
\p_{i}\big( X > 3\lambda, Y < \delta \lambda \,\big|\, X > \lambda\big) &=
\p_{i}\big( T_{\lfloor 3\lambda \rfloor +1} \leq T , m_ia_{ii}(T) < \delta \lambda \,\big|\, T_{\lfloor\lambda\rfloor+1} \leq T \big)  \\
& \le \p_{i}\big(T_{\lfloor 3\lambda\rfloor - \lfloor\lambda\rfloor} \leq m_ia_{ii}^{-1}(\delta \lambda)\big)
\leq  \p_{i}\big(L^{i}([0,m_ia_{ii}^{-1}(\delta \lambda)]) \ge \lfloor2\lambda\rfloor\big).
\end{align*}
By Markov's inequality the last expression above can be bounded by
$$\lfloor2\lambda\rfloor^{-1}\,\E_{i}\left[L^{i}([0, m_ia_{ii}^{-1}(\delta \lambda)])\right]
= \lfloor2\lambda\rfloor^{-1}\,m_ia_{ii}\left(m_ia_{ii}^{-1}(\delta \lambda)\right) \leq \delta \frac{\lambda}{\lfloor 2\lambda\rfloor} ,$$
which is smaller than $\delta$, as required.
\end{proof}

We define $T_0=0$ and $T_k=\min\{ n>T_{k-1} \colon X_n=i\}$, for $k\geq 1$. Recall that
\smash{$\E_i L^j([T_{k-1},T_k))=1/m_i$} and hence, by the strong Markov property, the random variables
\smash{$\xi_k:=1-m_i \,L^\nu([T_{k-1},T_k))$}
are independent and identically distributed with mean zero. By Lemma~\ref{firstpassage}\,$(a)$
the first passage time of zero for this walk satisfies \smash{$\E_i[N^{1/2}]=\infty$}.
As \smash{$m_iL^{i}([0,T_*])\geq N-1$} the result follows.

\subsection{Proof of Theorem~\ref{T:moment_prop}\,$(ii)$.} We first prove the result in the simple case that the state space~$\s$
is finite. In this case the chain is positive recurrent and we have $a(n)\sim n$.%

\begin{lemma}\label{positive}
Suppose $\s$ is finite. Then, for any $i\in\s$, we have $\displaystyle\E_{i}\big[T_*^{\beta}\big] < \infty$, for all $0\leq \beta < \frac12$.%
\end{lemma}%

\begin{proof}
Let $T_0=0$ and,  for $k\in\N$, define $T_k=\min\{ n>T_{k-1} \colon X_n=i\}$.
Denote by $h_{ij}$ the probability that the chain started in~$i$ hits~$j$ before
returning to~$i$, and observe that irreducibility implies that $h_{ij}>0$. By the strong Markov property we
have $m_jL^j[0,T_1]=YZ$ where $Y$ is a Bernoulli variable with mean $h_{ij}$ and $Z$ is an independent geometric
with success parameter~$h_{ji}$. 
Hence $\E_i L^j[0,T_1)=h_{ij}/m_jh_{ji}$, which also equals $1/m_i$.
Recalling that \smash{$\E[Z^2]\leq2/h_{ji}^2$} we get \smash{$\E_i[L^j[0,T_1)^2]\leq 2h_{ij}/m_j^2h_{ji}^2$}, and hence
$L^\nu[0,T_1)$ has finite variance. Define
\smash{$\xi_k:=1-m_i\,L^\nu([T_{k-1},T_k)),$}
and observe that $\xi_1, \xi_2,\ldots$ are independent and identically distributed variables with mean zero and
finite variance. Let \smash{$N:=\min\{ n \colon \sum_{k=1}^n \xi_k\leq 0\}$},
and observe that \smash{$T_*\leq T_N$}. Fix $\eps>0$ and note that
\begin{align*}
\prob_i\big( T_*>n\big)
& \leq \prob_i\big( N>\eps n \big) + \prob_i\Big( \sum_{k=1}^{\lceil \eps n\rceil} (T_k-T_{k-1})>n\Big).
\end{align*}
By Lemma~\ref{firstpassage}\,$(b)$
the first term on the right-hand side is bounded by a constant multiple of $(\eps n)^{-1/2}$.
For the second term we note that the random variables $T_1-T_0, T_2-T_1,\ldots$ are independent and
identically distributed with finite variance. By Chebyshev's inequality we infer that, for sufficiently small
$\eps>0$, the term is bounded by a multiple of $1/n$. Altogether we get that
\smash{$\prob_i( T_*>n)$} is bounded by a constant multiple of \smash{$n^{-1/2}$}, from which the result follows immediately.
\end{proof}

We return to the general case. The next result, which is an auxiliary step in the proof
of Theorem~\ref{T:moment_prop}\,$(ii)$, may be of independent interest. The short proof given here,
which does  not make any regularity assumptions on the chain, is due to Vitali Wachtel.

\begin{lemma}\label{hitti}
Fix a state $i\in\skris$ and let $T=\min\{ n>0 \colon X_n=i\}$ be the first return time to this state.
Then $$\E_{i}\big[a_{ii}(T)^{\alpha}\big] < \infty, \quad \mbox{ for all $0\leq \alpha < 1$.}$$
\end{lemma}

\begin{proof}
By Lemma~1 in Erickson~\cite{E73}, we have for $m(n):=\int_0^n \prob_{i}(T>x)\, dx$ that
$$\frac{n}{m(n)}\leq m_ia_{ii}(n) \leq 2 \, \frac{n}{m(n)} \qquad \mbox{ for all positive integers } n.$$
As $m(n)\geq n\prob_{i}(T>n-1)$ we infer that $m_ia_{ii}(n)\leq 2/\prob_{i}(T>n-1)$ and therefore
$$\E_{i}\big[a_{ii}(T)^\alpha\big] \leq \big(\sfrac{2}{m_i}\big)^\alpha\sum_{n=1}^\infty \big(\prob_{i}(T>n-1) \big)^{-\alpha}\, \prob_{i}(T=n)
= \big(\sfrac{2}{m_i}\big)^\alpha \sum_{n=1}^\infty (1-s_{n-1})^{-\alpha} \big(s_n-s_{n-1}\big),$$
where $s_n:=\prob_{i}(T\leq n)$. Letting $s(t):=s_{n-1}+(t-(n-1))(s_n-s_{n-1})$, for $n-1\leq t<n$, we
can bound the sum by $\int_0^\infty (1-s(t))^{-\alpha} \, ds(t)$, which is finite for all $0\leq\alpha<1$,
as required.
\end{proof}

We now look at the reduction of our Markov chain to the finite state space $\skris'=\{0\}\cup\{j\in\skris \colon \nu_j>0\}$. More explicitly, let
$t_0=0$ and $t_k=\min\{n>t_{k-1} \colon X_n\in\skris'\}$ for $k\in\N$, and $t_k=\max\{n<t_{k+1} \colon X_n\in\skris'\}$ for $k\in-\N$. Then $Y_n=X_{t_n}$
defines an irreducible Markov chain $Y=(Y_n \colon n\in\Z)$ with finite state space $\skris'$, and its invariant measure is $(m_i \colon i\in \skris')$.
If $N$ is the stopping time constructed in Theorem~2 for the reduced chain $Y$, then the solution~$T_*$ for the original
problem  is~\smash{$T_*=t_{N}.$}
\smallskip

Given two states $i,j\in \s'$ we denote by $S_{ij}$ a random variable whose law is given by
${\mathsf P}(S_{ij}= s)=\prob_i(t_1=s \, | \, Y_1=j)$ for all $s\in\N$, if $\prob_i(Y_1=j)>0$, and $S_{ij}=0$ otherwise.
We construct a probability space 
on which there are independent families $(S_{ij}, S_{ij}^{\ssup k} \colon k\in\N)$ of independent random variables with this law,
together with an independent copy of~$Y$ and hence~$N$.  We denote probability and expectation on this space by ${\mathsf P}$, resp.~${\mathsf E}$.
Observe that on this space we can also define a copy of the process $(t_k \colon k\in\N)$ by $t_0=0$ and $$t_k=t_{k-1}+\sum_{i,j\in\s'}
S_{ij}^{\ssup k}\1\{Y_{k-1}=i, Y_k=j\}\qquad \mbox{ for }k\in\N.$$
For any non-decreasing, subadditive representative $a$ of the class of the asymptotic Green's function,
$$\E_{i}\big[a(T_*)^{\beta}\big]
= {\sf E} \Big[ a\Big( \sum_{k=1}^{N} t_k-t_{k-1}\Big)^{\beta}\Big]
\leq {\sf E} \Big[ a\Big(  \sum_{k=1}^{N} \sum_{i,j\in\s'} S_{ij}^{\ssup k}\Big)^{\beta}\Big]
\leq \sum_{i,j\in\s'}  {\sf E}\Big[  a\Big( \sum_{k=1}^{N}  S_{ij}^{\ssup k}\Big)^{\beta}\Big].$$
It therefore suffices to show that
$${\sf E}\Big[  a_{ii}\Big( \sum_{k=1}^{N}  S_{ij}^{\ssup k}\Big)^{\beta}\Big]<\infty.$$
Let $n\in\N$ and use first subadditivity of $a_{ii}$ and then Jensen's inequality to get, for $2\beta<\alpha<1$, that
$${\sf E}\Big[  a_{ii}\Big( \sum_{k=1}^{n}  S_{ij}^{\ssup k}\Big)^{\beta}\Big]
\leq {\sf E}\Big[  \Big( \sum_{k=1}^{n}a_{ii}^\alpha\big(S_{ij}^{\ssup k}\big) \Big)^{\beta/\alpha}\Big]
\leq \Big(  \sum_{k=1}^{n} {\sf E} \big[a_{ii}^\alpha\big(S_{ij}^{\ssup k}\big)\big] \Big)^{\beta/\alpha}
= n^{\beta/\alpha} {\sf E} \big[a_{ii}^\alpha\big(S_{ij}\big)\big]^{\beta/\alpha}  .$$
We now note that, if $T_{ij}$ denotes the first hitting time of state~$j$ for $X$ under $\prob_i$, we have
${\mathsf P}(S_{ij}>x)\leq C_0\, \prob_i(T_{ij}>x)$ for all $x>0$, where $C_0$ is the maximum of the inverse of all
nonzero transition probabilities from $i$ to all other states, by the chain~$Y$. Hence
$${\sf E}\big[a_{ii}^\alpha\big(S_{ij}\big)\big] \leq C_0\, {\E_i}\big[a_{ii}^\alpha\big(T_{ij}\big)\big].$$
In the case $i=j$ the right hand side is finite by Lemma~\ref{hitti} and, as $a_{ii}$ grows no faster than linearly,
the right hand side is finite for all choices of $i,j\in\s'$ by application of Theorem~1.1 in Aurzada et al.~\cite{OM11}.
Summarising, we have found a constant $C>0$ such that
$${\sf E}\Big[  a_{ii}\Big( \sum_{k=1}^{n}  S_{ij}^{\ssup k}\Big)^{\beta}\Big]\leq C n^{\beta/\alpha}.$$
Using the independence of $N$ and $(S_{ij}^{\ssup k} \colon k\in\N)$ and Lemma~\ref{positive} we get
$${\sf E}\Big[  a_{ii}\Big( \sum_{k=1}^{N}  S_{ij}^{\ssup k}\Big)^{\beta}\Big]
\leq C \E_i\big[N^{\beta/\alpha}\big]<\infty,$$
as required.

\section{Optimality of $T_{*}$: Proof of Theorem~4}\label{S:Optimality}

In this section we prove Theorem~\ref{T:Optimality}. We start by introducing an intuitive and convenient way to talk about allocation rules. A path of the Markov chain $X$ can be viewed as leaving white and couloured balls on the integers, in the following way: At each site $k \in \Z$ we place one white ball if $X_{k}=i$, and \smash{$\frac{m_{i}}{m_{j}}\nu_{j}$} balls of colour~$j$ if $X_{k}=j$. By our assumption there is always an integer number of balls at each site. We call a bijection from the set of white balls to the set of coloured balls a \emph{matching}. Given a matching  we define an allocation rule $\tau\colon \Omega \times\Z\to\Z$ by letting
\begin{itemize}
\item$\tau(k)=k$ if there is no white ball at site~$k$,
\item $\tau(k)=\ell$ if the white ball at site $k$ is matched to a coloured ball at site~$\ell$.
\end{itemize}
Every allocation rule thus constructed balances $L^\mu$ and $L^\nu$, for $\mu=\delta_i$.
Conversely, every balancing
allocation rule agrees $L^\mu$-almost everywhere with an allocation rule constructed from a matching.
We denote by $\tau_*\colon \Omega \times\Z\to\Z$ the allocation rule associated with $T_*$
constructed in Proposition~\ref{T:existence_allocation_rule_extended}.
\medskip%

The allocation rule $\tau_*$ is associated with the following one-sided \emph{stable matching} or \emph{greedy algorithm},
which is a variant of the famous Gale--Shapley stable marriage algorithm~\cite{GS62}.
\begin{itemize}
\item[(1)] If the next occupied site to the right of a white ball carries one or more  coloured balls, map the
white ball to one of those coloured balls.
\item[(2)] Remove all white and coloured balls used in step (1) and repeat.
\end{itemize}
By Lemma~\ref{prob} the algorithm matches every ball after a finite number of steps, and it is easy to see that this
leads to the allocation rule~$\tau_*$.
\smallskip

Now recall from Section~2 that non-negative, possibly randomized, times~$T$ are associated to 
transport rules $\theta\colon \Omega\times\Z\times\Z\to[0,1]$ balancing $L^\mu$ and $L^\nu$ with 
the property that $\theta_\omega(x,y)=0$ whenever $x>y$. Without loss of generality 
we may assume that $\theta_\omega(x,x)=1$ if the site $x$ does not carry a white ball. This implies
that, for $x<y$, we can have $\theta_\omega(x,y)>0$ only if the site $x$ carries a white ball, 
and  the site $y$ carries a coloured ball. Moreover, if $y$ carries a ball of colour $j$, we have
$$\sum_{x<y} \theta_\omega(x,y) = \frac{m_i}{m_j}\nu_j.$$
Suppose that $u,v\in\Z$ with $u<v$. We say that the pair $(u,v)$ is \emph{crossed} by $\theta$ if there exist sites $x<u<v<y$ 
such that  $\theta(x,v)>0$ and $\theta(u,y)>0$. In this case $(x,u,v,y)$ is called a \emph{crossing}. 
\pagebreak[3]

For a transport rule $\theta$ we repair the crossing~$(x,u,v,y)$ by letting
\begin{itemize}
\item $\theta'(x,y)=\theta(x,y)+ (\theta(x,v) \wedge \theta(u,y))$,
\item $\theta'(u,v)=\theta(u,v)+ (\theta(x,v) \wedge \theta(u,y))$,
\item $\theta'(x,v)=\theta(x,v)- (\theta(x,v) \wedge \theta(u,y))$,
\item $\theta'(u,y)=\theta(u,y)- (\theta(x,v) \wedge \theta(u,y))$,
\end{itemize}
and setting $\theta'(w,z)=\theta(w,z)$ if $w\not\in\{x,u\}$ or $z\not\in\{y,v\}$, see Figure~\ref{f1}.  Note that
$\theta'$ is still a transport rule, the crossing has been repaired, i.e.\ $(x,u,v,y)$ 
is not a crossing by $\theta'$, and if $\theta$ balances $L^\mu$ and~$L^\nu$ then so does $\theta'$.

We now explain how to repair a pair $(u,v)$ crossed by $\theta$ by sequentially repairing its crossings and taking limits, so that $(u,v)$ is not crossed by the limiting transport rule. For this purpose we
define that a sequence of transport rules $\theta_n$ \emph{converges uniformly} to a transport rule 
$\theta$ if $$\lim_{n\to\infty} \sum_{x,y\in\Z} \big| \theta_n(x,y)-\theta(x,y) \big|=0. $$
Denote by $y_1, y_2, \ldots$ the 
sequence of sites $v<y_1< y_2<\cdots$ such that $\theta(u, y_n)>0$, and by $x_1, x_2,\ldots$ the 
sequence of sites $u>x_1> x_2>\cdots$ such that $\theta(x_n, v)>0$. Note that both sequences could be finite or infinite. First we successively repair the crossings $x_1<u<v<y_n$, for $n=1,2,\ldots$. 
The total mass moved in the $n$th repair is bounded by $4\theta(u,y_n)$ and because $\sum_n \theta(u,y_n)\leq 1$ we can infer that the sequence of 
repaired transport rules converges uniformly to a transport rule $\theta_1$. Of course, here and below if a sequence is finite we take the last  element of the sequence as limit. We continue by repairing the crossings $x_2<u<v<y_n$ of $\theta_1$, for $n=1,2,\ldots$, 
obtaining $\theta_2$, and so on. We obtain a sequence $\theta_1,\theta_2,\ldots$ 
of transport rules. The amount of mass moved when going from $\theta_{n-1}$ to $\theta_{n}$ is bounded by $4\theta(x_n,v)$. As $\sum_n \theta(x_n,v)<\infty$,
we infer that the sequence $(\theta_n)_n$ converges uniformly to a limiting transport rule. We observe that this transport rule  balances $L^\mu$ and $L^\nu$ and that $(u,v)$ is not crossed by it.

\begin{figure}
\begin{tikzpicture}
[touch/.style={circle, shading=ball, ball color=white, inner sep=0.5mm, minimum size=0.2cm},
C1Touch/.style={circle, shading=ball, ball color=black, inner sep=0.5mm, minimum size=0.2cm},
location/.style={inner sep=0mm}
]
\draw (-8,0) to (0,0);
\foreach \time/\name in {%
-15.5/w1, -11.5/w2
}
{
\node[touch] (\name) at (0.5*\time,0) {};
};
\foreach \time/\name in {%
-6.0/w3,-1.0/w4%
}
{
\node[C1Touch] (\name) at (0.5*\time,0) {};
}
\draw (-8,-4) to (0,-4);
\foreach \time/\name in {%
{-15.5,-4}/w21, {-11.5,-4}/w22
}
{
\node[touch] (\name) at (0.5*\time,0) {};
};
\foreach \time/\name in {%
{-6.0,-4}/w23,{-1,-4}/w24%
}
{
\node[C1Touch] (\name) at (0.5*\time,0) {};
}
\draw (-7.75,-0.4) node[rectangle] {$x$};
\draw (-5.75,-0.4) node[rectangle] {$u$};
\draw (-3.0,-0.4) node[rectangle] {$v$};
\draw (-0.5,-0.4) node[rectangle] {$y$};
\draw (0.2,0) node[rectangle] {$\mathbb{Z}$};
\draw (-7.75,-4.4) node[rectangle] {$x$};
\draw (-5.75,-4.4) node[rectangle] {$u$};
\draw (-3.0,-4.4) node[rectangle] {$v$};
\draw (-0.5,-4.4) node[rectangle] {$y$};
\draw (0.2,-4) node[rectangle] {$\mathbb{Z}$};
\draw (w1.north east) to [bend left=60] node[above, scale=.7]{$\theta(x,v)$} (w3.north west);
\draw (w2.north east) to [bend left=60] node[above, scale=.7]{$\theta(u,y)$} (w4.north west);
\draw[densely dashed] (w21.north east) to [bend left=70] node[above, scale=.7]{$\theta(x,y)$ $\boldsymbol{+\theta_{\min}}$} (w24.north west);
\draw[densely dashed] (w22.north east) to [bend left=60] node[below=3pt, scale=.7]{$\theta(u,v)$ $\boldsymbol{+\theta_{\min}}$} (w23.north west);
\draw[loosely dotted] (w21.north east) to [bend left=60] node[above, scale=.7]{$\theta(x,v)$ $\boldsymbol{-\theta_{\min}}$} (w23.north west);
\draw (w22.north east) to [bend left=60] node[above, scale=.7]{$\theta(u,y)$ $\boldsymbol{-\theta_{\min}}$} (w24.north west);
\end{tikzpicture}
\caption{The picture above shows a crossing. Its weight
${\theta_{\min}}:=\theta(x,v)\wedge\theta(u,y)$ is assumed
 to be $\theta(x,v)$, so that in the picture below we see that after the repair
 the dotted edge has weight zero, and the crossing is therefore removed.}
 \label{f1}
\end{figure}
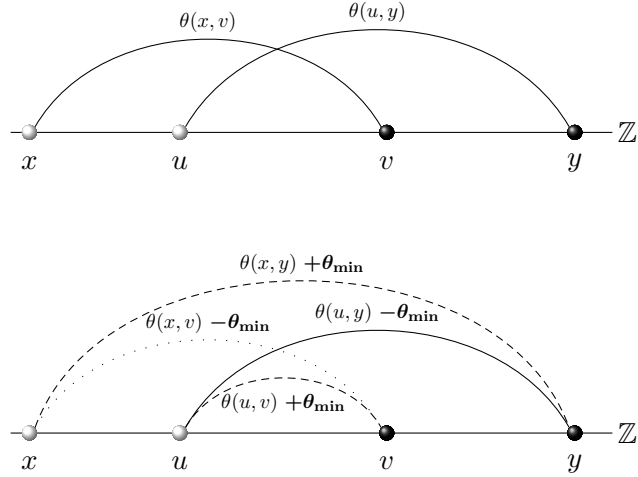

\begin{lemma}\label{5.1}
Suppose that $\theta$  is a transport rule balancing $L^\mu$ and $L^\nu$ and $A\subset\Z$ a finite interval. Then, by repairing pairs crossed by $\theta$ in a given order, 
we obtain a transport rule $\theta_*$ 
balancing $L^\mu$ and~$L^\nu$, such that  if $u,v\in A$ then 
$(u,v)$ is not crossed by $\theta_*$.
\end{lemma}

\begin{proof}
Without loss of generality the left endpoint of $A$ carries a white ball, and its right endpoint carries a coloured ball. Let $v_1,\ldots,v_n$ be the sites in $A$ carrying coloured balls, ordered from left to right. We go through these sites in order, starting with $v_1$. Take $u_1$ to be the rightmost site
to the left of $v_1$ carrying a white ball. Repair the pair $(u_1,v_1)$ as above, and observe that the resulting transport rule transports a unit mass from $u_1$ to~$v_1$. We declare the white ball at site $u_1$ and one of the coloured balls at~$v_1$ \emph{cancelled}. If $v_1$ carries an uncancelled ball and there are uncancelled white balls on sites of $A$ to the left of $v_1$, we choose the rightmost of those, say $u_2$, 
repair the pair $(u_2, v_1)$, and cancel two balls as above. We continue until we run out of uncancelled balls.
The resulting transport rule has the property that none of the pairs $(u,v_1)$, with $u\in A$, is crossed, and from all sites carrying cancelled white balls a unit mass is transported to site $v_1$. 

We now move to the next coloured ball $v_2$ and repair all pairs $(u,v_2)$, where $u$ goes from right to left through all sites in $A\cap (-\infty,v_2)$ carrying uncancelled white balls. We do this until we run out of uncancelled white balls to the left of,  or coloured balls on the site~$v_2$. Observe that at the end of this step none of the pairs $(u,v_1)$ or $(u,v_2)$, with $u\in A$, is crossed by the resulting transport rule.
We continue, moving to the next coloured ball until all coloured balls in $A$ are exhausted. 
At the end of this finite procedure we obtain a transport rule $\theta_*$ 
balancing $L^\mu$ and $L^\nu$, such that  if $u,v\in A$ then $(u,v)$ is not crossed by $\theta_*$.  
\end{proof}

We call a set $A$ an \emph{excursion} if it is an interval $[m, n]$ such that that there is the same number of  white and coloured balls on the sites of~$A$,  but the number of white balls exceeds the number of coloured balls on every subinterval $[m,k]$, for $m\leq k<n$. Observe that if $A$ is an excursion, then it is an interval of the form $[m,\tau_*(m)]$ where $m$ carries a white ball, but not all such intervals are excursions. Moreover, for every $x\in A$, we have both  $\tau_*(x)\in A$ and  $\tau_*^{-1}(x)\subset A$.
\medskip

\begin{lemma}\label{iden}
Let $A$ be an excursion and \smash{$\theta_*$} a transport rule balancing $L^\mu$ and $L^\nu$, such that
any pair $(u,v)$ with $u,v\in A$ is not crossed by~$\theta_*$. Then $\theta_*$  agrees  in~$A$ with the 
allocation rule~$\tau_*$, in the sense that
$\theta_*(x,y)=\1\{\tau_*(x)=y\}$ and $\theta_*(y,x)=\1\{\tau_*(y)=x\}$, for all $x\in A$ and $y\in\Z$.
\end{lemma}

\begin{proof} We start by fixing a site $x\in A$ carrying a 
white ball, and note that, by definition of an excursion, we also have $\tau_*(x)\in A$. We show by contradiction that $\theta_*$ transports no
mass from $x$ to a point other than~$\tau^*(x)$.

First, suppose that there exist $x<v<\tau_*(x)$ with $\theta_*(x,v)>0$. As there are more white than coloured balls on the sites in $[x,v]$, and as every site carries at most 
one white ball, we find $x'\in(x,v)$ such that the sites of $[x',v]$ carry the same number of white and coloured balls. As $\theta_*(x,v)>0$ not all white balls in $[x',v]$
are matched within that interval, and there must also exist $u\in[x',v)$ and
$y>v$ such that $\theta_*(u,y)>0$. So we have found a pair $(u,v)$ with $u,v\in A$, which is crossed by $\theta_*$, 
and hence a contradiction.

Second, suppose that there exist $v>\tau_*(x)$ with $\theta_*(x,v)>0$. As there are at least as many coloured balls as white balls in 
$[x,\tau_*(x)]$ not all coloured balls are matched within that interval, and hence there exists a $y\in(x,\tau_*(x)]$ and a site $u< x$ 
with $\theta_*(u,y)>0$. So we have found a pair $(x,y)$ with $x,y\in A$, which is crossed by~$\theta_*$, and hence a contradiction.
We conclude that $\theta_*(x,y)=\1\{\tau_*(x)=y\}$  for all $x\in A$.

Now fix a site $x\in A$ carrying balls of colour~$j$. Then $\tau_*^{-1}(x)$ is a set of
\smash{$(m_i/m_j)\nu_j$} points in $A$. Hence, by the first part, \smash{$\theta_*(y,x)=\1\{\tau_*(y)=x\}$} for all $y\in\tau_*^{-1}(x)$. Moreover, 
$$\sum_{y\in\tau_*^{-1}(x)} \theta_*(y,x) =(m_i/m_j)\nu_j=\sum_{y\in\Z} \theta_*(y,x).$$ 
Hence $\theta_*(y,x)=0=\1\{\tau_*(y)=x\}$ also for all $y\not\in\tau_*^{-1}(x)$.
\end{proof}

We now let $\psi$ be a non-negative, concave function on the non-negative integers~$\N_{0}$.
Note that this implies that $\psi\colon \N_0\to[0,\infty)$ is non-decreasing. We further assume that
$\psi(0)=0$, an assumption which causes no loss of generality in Theorem~\ref{T:Optimality}.
We write $\psi(n)=0$ for $n\leq0$ to simplify the notation.%
\smallskip%

\begin{lemma}\label{opti}
Let $A$ be an excursion and suppose $\theta$ is a transport rule balancing $L^\mu$ and $L^\nu$. Then\\
$$
\sum_{x\in A} \psi\big(\tau_*(x)-x\big) + \sum_{x\in \tau_*^{-1}(A)} \psi\big(\tau_*(x)-x\big)
\leq \sum_{\heap{x\in A}{y \in\Z}} \theta(x,y)\,\psi(y-x) + \sum_{\heap{x \in \Z}{y\in A}} \theta(x,y)\,\psi(y-x) .
$$
\end{lemma}

\begin{proof}
Observe that, by concavity, for all $a,b,c \in \N_{0}$, we have
\begin{equation}\label{impro}
\psi(a+b)+\psi(b+c) \ge \psi(a+b+c) + \psi(b).\\[2mm]
\end{equation}
Fix a crossing $x<u<v<y$  with $u,v\in A$, and let $\theta'$ be the
result of repairing the crossing. We show that repairing the crossing does not increase
$$\sum_{\heap{x\in A}{y \in\Z}} \theta(x,y)\,\psi(y-x) +
\sum_{\heap{x \in \Z}{y\in A}} \theta(x,y)\,\psi(y-x) $$
by looking at the difference of the repaired and original state of the sum.
If $x,y\not\in A$ we get
\begin{align*}
\theta'(u,y)&\psi(y-u)+ 2\theta'(u,v)\psi(v-u)+\theta'(x,v)\psi(v-x)\\
&\phantom{xxx} - \big(\theta(u,y)\psi(y-u)+ 2\theta(u,v)\psi(v-u)+\theta(x,v)\psi(v-x)\big)\\
& = \big( \theta(x,v)\wedge \theta(u,y)\big) \big( 2\psi(v-u)-\psi(v-x)-\psi(y-u)\big) \leq 0,
\end{align*}
as $\psi$ is non-decreasing. If $x\in A, y\not\in A$ we get
\begin{align*}
\theta'(u,y)&\psi(y-u)+ 2\theta'(u,v)\psi(v-u)+2\theta'(x,v)\psi(v-x)+\theta'(x,y)\psi(y-x)\\
&\phantom{xxx} - \big(\theta(u,y)\psi(y-u)+ 2\theta(u,v)\psi(v-u)+2\theta(x,v)\psi(v-x)+\theta(x,y)\psi(y-x)\big)\\
& = \big( \theta(x,v)\wedge \theta(u,y)\big) \big( 2\psi(v-u)+\psi(y-x)-2\psi(v-x)-\psi(y-u)\big) \\
& \leq  \big( \theta(x,v)\wedge \theta(u,y)\big) \big( \psi(v-u)+\psi(y-x)-\psi(v-x)-\psi(y-u)\big) \leq 0,
\end{align*}
using first that $\psi$ is non-decreasing and then~\eqref{impro}. The case $x\not\in A, y\in A$ is analogous.
If $x,y\in A$ the difference is twice
\begin{align*}
\theta'(x,v)&\psi(v-x)+\theta'(u,y)\psi(y-u)+ \theta'(u,v)\psi(v-u)+\theta'(x,y)\psi(y-x)\\
&\phantom{xxx} - \big(\theta(x,v)\psi(v-x)+\theta(u,y)\psi(y-u)+ \theta(u,v)\psi(v-u)+\theta(x,y)\psi(y-x) \big)\\
& = \big( \theta(x,v)\wedge \theta(u,y)\big) \big( \psi(y-x)+\psi(v-u)-\psi(v-x)-\psi(y-u)\big) \leq 0,
\end{align*}
by application of~\eqref{impro}, 
which shows that in all cases the sum above is not increased by the repair.
\pagebreak[3]

Repairing crossings successively as described in Lemma~\ref{5.1}, we get
$$
\sum_{\heap{x\in A}{y \in\Z}} \theta_*(x,y)\,\psi(y-x) + \sum_{\heap{x \in \Z}{y\in A}} \theta_*(x,y)\,\psi(y-x)
\leq \sum_{\heap{x\in A}{y \in\Z}} \theta(x,y)\,\psi(y-x) + \sum_{\heap{x \in \Z}{y\in A}} \theta(x,y)\,\psi(y-x) .
$$
By Lemma~\ref{iden} we have $\theta_*(x,y)=\1\{\tau(x)=y\}$ if $x\in A$ or $y\in A$, and this allows us to rewrite the left hand side as stated.
\end{proof}
\smallskip

\begin{lemma}\label{ergo}
Let $T\geq 0$ be a (possibly randomized) unbiased shift and $\theta\colon \Omega \times\Z\times\Z\to[0,1]$ be the associated transport rule. Denote by $(T_n \colon n\in\Z)$
the times in which $X$ visits the state~$i$, in order so that $T_0=0$. Let $\psi\colon \Z \to [0,\infty)$
be concave.  Then, $\prob_i$-almost surely,
$$\lim_{n\to\infty} \frac1n \Big\{\sum_{k=0}^{T_n-1} \sum_{\ell=k+1}^\infty \theta(k,\ell) \psi(\ell-k) +
\sum_{\ell=-\infty}^{k-1}  \theta(\ell,k)  \psi(k-\ell)\Big\}
= 2\E_i ^{_{^\oplus}}\psi(T),$$
and
$$\lim_{m\to\infty} \frac1m \Big\{\sum_{k=T_{-m}}^1 \sum_{\ell=k+1}^\infty \theta(k,\ell) \psi(\ell-k)
+ \sum_{\ell=-\infty}^{k-1} \theta(\ell,k) \psi(k-\ell)\Big\}
= 2\E_i^{_{^\oplus}} \psi(T).$$
\end{lemma}

\begin{proof}
We observe, from the strong Markov property, that $\xi_n=(X_{T_{n-1}+1}, \ldots, X_{T_n})$,
$n\in\Z$, are independent and identically distributed random vectors. Hence their shift is stationary and ergodic, see for example~\cite[8.4.5]{Du02}. By the ergodic theorem, see e.g.~\cite[8.4.1]{Du02}, $\prob_i$-almost surely,
\begin{align*}
\lim_{n\to\infty} \frac1n \Big\{\sum_{k=0}^{T_n-1} \sum_{\ell=k+1}^\infty \theta(k,\ell) \psi(\ell-k) \Big\}
= \E_i \sum_{\ell=1}^\infty \theta_\omega(0,\ell) \psi(\ell)= \E_i^{_{^\oplus}} \psi(T).
\end{align*}
Similarly,
\begin{align*}
\lim_{n\to\infty} \frac1n \Big\{ \sum_{k=0}^{T_n-1} \sum_{\ell=-\infty}^{k-1}  \theta(\ell,k)  \psi(k-\ell)
\Big\} =
\E_i \sum_{k=0}^{T_1-1} \sum_{\ell=-\infty}^{k-1} \theta_\omega(\ell,k) \psi(k-\ell).
\end{align*}
The expectation equals
$$\sum_{j\in\s} \frac{m_j}{m_i}\,
\E_j \sum_{\ell=-\infty}^{-1} \theta_\omega(\ell,0) \psi(-\ell)= \frac{1}{m_i} \E \sum_{\ell=-\infty}^{-1} \theta_\omega(\ell,0) \psi(-\ell)
= \frac1{m_i} \E \sum_{\ell=1}^\infty \theta_\omega(0,\ell) \psi(\ell)=\E_i^{_{^\oplus}}\psi(T),$$
using translation invariance of $\theta$.
The second statement follows in the same manner.
\end{proof}


\begin{figure}\begin{center}\scalebox{0.9}{
\begin{tikzpicture}
\draw[->] (-5,0) -- (11,0) node[right] {$n$};
\draw[-] (0,-4.5) -- (0,4.5); 
\draw[] (8,4) node[] {\footnotesize{$L^{\mu}([0,n))-L^{\nu}([0,n)) \,\,(n\geq 0)$}};
\draw[] (-3,4) node[] {\footnotesize{$-L^{\mu}([n,0))+L^{\nu}([n,0)) \,\,(n<0)$}};

\foreach \x/\xtext in {-4.7/T_{-5}, -4.1/T_{-4}, -3.6/T_{-3}, -2.9/T_{-2}, -1.4/T_{-1},
0/T_{0}, 1.4/T_{1}, 2.3/T_{2}, 5.5/T_{3}, 7/T_{4}, 7.8/T_{5}, 10.5/T_{6} }
\draw[shift={(\x,0)}] (0pt,2pt) -- (0pt,-2pt) node[above=3pt, scale=0.7] {$\xtext$};

\node (s1) at (-4.7,-3) [fill, circle, inner sep=1.2pt] {};
\draw [thick] (-5,-3) -- (-4.7,-3);

\node (h2) at (-4.7,-2) [draw, circle, inner sep=1.2pt] {};
\node (s2) at (-4.1,-2) [fill, circle, inner sep=1.2pt] {};
\draw [thick] (h2) to (s2);

\node (h3) at (-4.1,-1) [draw, circle, inner sep=1.2pt] {};
\node (s3) at (-3.6,-1) [fill, circle, inner sep=1.2pt] {};
\draw [thick] (h3) to (s3);

\node (h4) at (-4.1,-1) [draw, circle, inner sep=1.2pt] {};
\node (s4) at (-3.6,-1) [fill, circle, inner sep=1.2pt] {};
\draw [thick] (h4) to (s4);

\node (h5) at (-3.6,0) [draw, circle, inner sep=1.2pt] {};
\node (s5) at (-2.9,0) [fill, circle, inner sep=1.2pt] {};
\draw [thick] (h5) to (s5);

\node (h6) at (-2.9,1) [draw, circle, inner sep=1.2pt] {};
\node (s6) at (-1.4,1) [fill, circle, inner sep=1.2pt] {};
\draw [thick] (h6) to (s6);

\node (h7) at (-1.4,2) [draw, circle, inner sep=1.2pt] {};
\node (s7) at (-0.5,2) [fill, circle, inner sep=1.2pt] {};
\draw [thick] (h7) to (s7);

\node (h8) at (-0.5,0) [draw, circle, inner sep=1.2pt] {};
\node (s8) at (0,0) [fill, circle, inner sep=1.2pt] {};
\draw [thick] (h8) to (s8);

\node (h9) at (0,1) [draw, circle, inner sep=1.2pt] {};
\node (s9) at (1.4,1) [fill, circle, inner sep=1.2pt] {};
\draw [thick] (h9) to (s9);

\node (h10) at (1.4,2) [draw, circle, inner sep=1.2pt] {};
\node (s10) at (2.3,2) [fill, circle, inner sep=1.2pt] {};
\draw [thick] (h10) to (s10);

\node (h11) at (2.3,3) [draw, circle, inner sep=1.2pt] {};
\node (s11) at (3,3) [fill, circle, inner sep=1.2pt] {};
\draw [thick] (h11) to (s11);

\node (h12) at (3,-1) [draw, circle, inner sep=1.2pt] {};
\node (s12) at (4.5,-1) [fill, circle, inner sep=1.2pt] {};
\draw [thick] (h12) to (s12);

\node (h13) at (4.5,-2) [draw, circle, inner sep=1.2pt] {};
\node (s13) at (5.5,-2) [fill, circle, inner sep=1.2pt] {};
\draw [thick] (h13) to (s13);

\node (h14) at (5.5,-1) [draw, circle, inner sep=1.2pt] {};
\node (s14) at (7,-1) [fill, circle, inner sep=1.2pt] {};
\draw [thick] (h14) to (s14);

\node (h15) at (7,0) [draw, circle, inner sep=1.2pt] {};
\node (s15) at (7.8,0) [fill, circle, inner sep=1.2pt] {};
\draw [thick] (h15) to (s15);

\node (h16) at (7.8,1) [draw, circle, inner sep=1.2pt] {};
\node (s16) at (8.8,1) [fill, circle, inner sep=1.2pt] {};
\draw [thick] (h16) to (s16);

\node (h17) at (8.8,-1) [draw, circle, inner sep=1.2pt] {};
\node (s17) at (9.7,-1) [fill, circle, inner sep=1.2pt] {};
\draw [thick] (h17) to (s17);

\node (h18) at (9.7,-3) [draw, circle, inner sep=1.2pt] {};
\node (s18) at (10.5,-3) [fill, circle, inner sep=1.2pt] {};
\draw [thick] (h18) to (s18);

\node (h19) at (10.5,-2) [draw, circle, inner sep=1.2pt] {};
\draw [thick] (h19) -- (11,-2);
 
\draw[gray, dashed, very thin] (-5,-3) -- (11,-3);
\draw[gray, dashed, very thin] (-5,-2) -- (11,-2); 
\draw[gray, dashed, very thin] (-5,-1) -- (11,-1);
\draw[gray, dashed, very thin] (-5,1) -- (11,1);
\draw[gray, dashed, very thin] (-5,2) -- (11,2);
\draw[gray, dashed, very thin] (-5,3) -- (11,3);

\draw[gray, very thin] (-4.7,0) -- (-4.7,-3.5) node[below] {$\sigma_{3}$};
\draw[gray, very thin] (-4.1,0) -- (-4.1,-3.5) node[below right] {$\sigma_{2}=\sigma_{1}$};
\draw[gray, very thin] (0,0) -- (0,-3.5) node[below] {$\tau_{0}$};
\draw[gray, very thin] (5.5,0) -- (5.5,-3.5) node[below] {$\tau_{1}=\tau_{2}$};
\draw[gray, very thin] (10.5,0) -- (10.5,-3.5) node[below] {$\tau_{3}$};

\draw[decorate,decoration={brace,mirror}]
(4, 2) -- (4, 3) node [black,midway,xshift=0.8cm] {$\frac{1}{m_{i}}$};

\end{tikzpicture}}\end{center}
\caption{A possible profile of local time differences over the excursion $[\sigma_{3},\tau_{3}-1]$.
Upward jumps  are of size \smash{$1/m_{i}$}, downward jumps are a positive integer multiple of \smash{$1/m_{i}$}, 
the actual value depending on the colour of the ball at the location of the jump.}
\label{f2}
\end{figure}
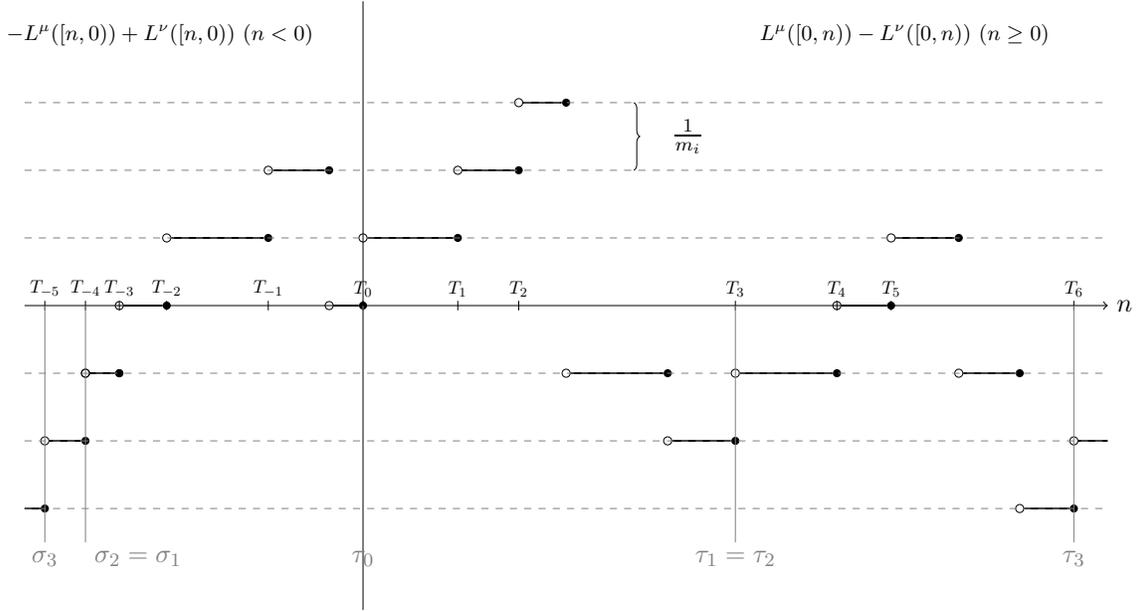

\begin{proof}[Proof of Theorem~\ref{T:Optimality}]
We now look at the sequence
$$\tau_n=\min\big\{T_k\geq 0 \colon L^\mu([0,T_k))-L^\nu([0,T_k))\leq \sfrac{-n}{m_i}\big\}.$$
Let $d_n= L^\mu([0,\tau_n))-L^\nu([0,\tau_n))$ and define
$$\sigma_n=\max\big\{k\leq 0 \colon -L^\mu([k,0))+L^\nu([k,0)) =  d_n\big\}.$$
$(\tau_n)$ and $(\sigma_n)$ are well-defined subsequences of $(T_n\colon n\in\Z)$,
$\prob_i$-almost surely, by Lemma~\ref{prob}. Moreover, $\tau_n\uparrow \infty$, $\sigma_n\downarrow-\infty$ and
by construction $[\sigma_n,\tau_n-1]$ is an excursion, see Figure~\ref{f2}. By Lemma~\ref{opti}
$$
\sum_{k=\sigma_n}^{\tau_n-1}
\Big\{  \psi\big(\tau_*(k)-k\big) + \sum_{\ell\in \tau_*^{-1}(k)} \psi\big(\tau_*(\ell)-\ell\big)\Big\}
\leq \sum_{ \heap{ \sigma_{n} \le k \le \tau_{n}-1}{\ell \in \Z} } 
\theta(k,\ell)\psi(k-\ell) + \sum_{\heap{ \sigma_{n} \le \ell \le \tau_{n}-1}{k \in \Z} } \theta(k,\ell)\psi(k-\ell)  .
$$
Lemma~\ref{ergo} shows that the left hand side is asymptotically equivalent to $2m_i\,L^i([\sigma_n, \tau_n])\, \E_i \psi(T_*)$ and the right hand side to
$2m_i\,L^i([\sigma_n, \tau_n])\,  \E_i^{_{^\oplus}} \psi(T)$, from which we conclude that $\E_i \psi(T_*)\leq\E_i^{_{^\oplus}} \psi(T)$.
\end{proof}

\pagebreak[3]

\section{Concluding remarks and open problems}\label{S:Conclusion}

\textbf{Non-Markovian setting.} Theorem~\ref{T:characterization_thm_intro} and Theorem~\ref{T:Existence_allocation_rule_intro} 
remain valid in a more general non-Markovian setting. We require that under the
$\sigma$-finite measure $\prob$ the stochastic process~$X$, taking values in the countable
state space $\skris$, is stationary with a strictly positive stationary $\sigma$-finite measure 
$(m_{i}: i \in \s)$.
The probability measure $\p_i$ is then defined by conditioning $X$ on the event $\{X_0=i\}$.
We further require that, for every $i,j\in\s$, the random sets
\smash{$\{n\in\N \colon X_n=j\}$} and \smash{$\{n\in\N \colon X_{-n}=j\}$}  are  
infinite $\p_i$-almost surely. Then both theorems
carry over to this conditioned process. Further technical conditions are required to
generalize Lemma~\ref{ergo} and hence extend Theorem 4 to the non-Markovian setting.
Theorem~3 however fully exploits the Markov structure and cannot be generalized easily.

\textbf{General inital distribution.} Although our main focus is on the case where the initial distribution is the Dirac measure
$\delta_{i}$ for some $i \in \s$, the statements of Proposition~\ref{P:transport_rule} and~\ref{P:allocation_rule} allow general initial distributions~$\mu$.
By conditioning on the initial state one can see that a \emph{sufficient} condition for existence
of the solution is that the target measure $\nu$ admits a decomposition \smash{$\nu=\sum_{i\in\s} \nu^{\ssup i} \mu_i$}, where $\nu^{\ssup i}$ are probability
measures on~$\s$, such that $m_i \nu_j^{\ssup i}/m_j$ are integers for
all $i,j\in\skris$. We do not believe that this is also a \emph{necessary} condition.

\textbf{Randomized shifts.} If the target measure $\nu$ fails to satisfy the integer condition in Theorem~\ref{T:characterization_thm_intro}\,(b),
extra randomization is needed to solve the embedding problem. With extra randomness any target measure~$\nu$ may be embedded in a way similar to the
extra head schemes in~\cite{HP05}: Take a random variable $U \sim \text{Uniform}(0,1)$ and define
\begin{equation}
T_{\texttt{rand}}:=\min\Big\{n\ge 0 \colon L^{i}([0,n]) - \sum_{j \in \s}\nu_j\,L^{j}([0,n]) \le \sfrac{U}{m_i} \Big\}.
\end{equation}
Then $T_{\texttt{rand}}$ is an unbiased shift embedding~$\nu$. We see that if the integer condition holds, the sample value of $U$ becomes irrelevant and
we recover the non-randomized solution $T_*$ defined in Theorem~\ref{T:Existence_allocation_rule_intro}.
\pagebreak[3]

\textbf{Brownian motion and optimal shifts.} Last et al.~\cite{LMT14} discuss the Skorokhod embedding problem for a two-sided Brownian motion
$(B_t)_{t\in\R}$. In this context a random time~$T$ solves the embedding problem if \smash{$(B_{T+t}-B_T)_{t\in\R}$} is a standard two-sided Brownian motion
independent of $B_{T}$ and the law of $B_{T}$ is~$\nu$. They show that for \emph{any} target distribution $\nu$ not charging the origin the stopping
time  $T_*=\inf\{t>0 \colon L^0_t=L^\nu_t\}$, where $(L^x_t \colon t>0)$ is the process of local times at level~$x$ and \smash{$L^\nu_t:=\int L^x_t \, \nu(dx)$},
solves the embedding problem. They further show that every solution~$T$ that is a stopping time satisfies \smash{$\E [T^{_{\frac14}}]=\infty$} while under a mild
condition on~$\nu$ the constructed solution
\smash{$T_*$} satisfies \smash{$\E [T_*^{_{\beta}}]<\infty$} for all $\beta<\frac14$. The techniques of the present paper can be adapted to improve the results of~\cite{LMT14}
by showing that \smash{$\E [T^{_{\frac14}}]=\infty$} even for non-negative solutions which are not necessarily stopping times, and also to show a strong optimality
result similar to Theorem~\ref{T:Optimality}, i.e.~that $\E_0 \psi(T_*) \leq \E_0 \psi(T)$ simultaneously for all non-negative concave functions~$\psi$. These results will appear in the
forthcoming thesis~\cite{Re15}.
\smallskip

\textbf{Signed shifts.} The optimality result of Theorem~\ref{T:Optimality} cannot be extended easily to random times~$T$ that can take both
positive and negative values. Indeed, starting from such a solution~$T$ and associating an allocation rule~$\tau$ to it, we may still make local
improvements by repairing crossings, but now there is more than one way to repair a crossing and the optimal way to do this appears to involve nonlocal
choices. To get a feeling for the difficulties, we look at a two-sided stable matching strategy that at a first glance looks like a good candidate for
an optimal solution.  In the language of Section~\ref{S:Optimality} we match a coloured ball to a white ball if \emph{both} the coloured ball is the
nearest coloured ball to the white ball, and the white ball is the nearest  white ball to the coloured ball (resolving possible ties in some deterministic way).
Locally, the resulting allocation rule may be better or worse than the one coming from our one-sided stable matching. Consider, for example, configuration
of balls in the order white--coloured--white--coloured  placed at distances $a, b, c$ such that $b<a, c$.  The two-sided algorithm matches the middle balls
and, if other balls are sufficiently far away, the outer balls, which gives a contribution of $\psi(b)+\psi(a+b+c)$. One-sided stable matching matches the
first pair and the second pair and gives $\psi(a)+\psi(c)$, and each contribution could be smaller or larger depending on the relative size of $a,b,c$.
Even finding the optimal moment properties of signed shifts is an open problem.
\smallskip

\textbf{Random fields.} A vast open area of possible further research are embedding problems for multiparameter processes and random fields. In higher dimensions stable allocation procedures no longer have optimal moment properties, see for example~Holroyd, Peres and Schramm~\cite{HPPS09}, so other methods  need to be considered. It would be particularly interesting to investigate embedding problems for spin systems such as the infinite volume Gibbs measure of the Ising model at high temperature. 
\smallskip

{\bf Acknowledgments:} We would like to thank Vitali Wachtel for valuable discussions on Wiener-Hopf decompositions and for suggesting the proof of Lemma~\ref{hitti}, which allowed
us to remove a regularity condition on the asymptotic Green's function of the Markov chain. 
The first author would like to thank Mathias Beiglb\"ock and Martin Huesmann
for enlightening discussions during visits to the Hausdorff Institute and Eurandom. We would particularly like to thank Martin Huesmann for conjecturing the  result of
Theorem~\ref{T:Optimality}. Last but not least, we would like to thank G\"unter Last and an anonymous referee for several insightful comments.



\begin{thebibliography}{10}

\bibitem{OM11}
F.~Aurzada, H.~D{\"o}ring, M.~Ortgiese, and M.~Scheutzow.
\newblock Moments of recurrence times for {M}arkov chains.
\newblock {\em Electron. Commun. Probab.}, 16:296--303, 2011.

\bibitem{BY81}
M.~T. Barlow and M.~Yor.
\newblock ({S}emi-) martingale inequalities and local times.
\newblock {\em Z. Wahrsch. Verw. Gebiete}, 55:237--254, 1981.

\bibitem{Bass95}
R.~F. Bass.
\newblock {\em Probabilistic techniques in analysis}.
\newblock Probability and its Applications (New York). Springer-Verlag, New
  York, 1995.

\bibitem{BH13}
M.~Beiglb\"ock, A.M.G.~Cox, and M.~Huesmann.
\newblock Optimal transport and {S}korokhod embedding.
\newblock {\em arXiv:1307.3656}, 2013.


\bibitem{XC99}
X.~Chen.
\newblock How often does a {H}arris recurrent {M}arkov chain recur?
\newblock {\em Ann. Probab.}, 27:1324--1346, 1999.

\bibitem{Du02}
R.~M. Dudley.
\newblock {\em Real analysis and probability}, volume~74 of {\em Cambridge
  Studies in Advanced Mathematics}.
\newblock Cambridge University Press, Cambridge, 2002.

\bibitem{E73}
K.~B. Erickson.
\newblock The strong law of large numbers when the mean is undefined.
\newblock {\em Trans. Amer. Math. Soc.}, 185:371--381, 1973.

\bibitem{Feller}
W.~Feller.
\newblock {\em An introduction to probability theory and its applications.
  {V}ol. {II}.}
\newblock Second edition. John Wiley \& Sons, Inc., New York-London-Sydney,
  1971.

\bibitem{GS62}
D.~Gale and L.~S. Shapley.
\newblock College admissions and the stability of marriage.
\newblock {\em Amer. Math. Monthly}, 69:9--15, 1962.

\bibitem{HL01}
A.~E. Holroyd and T.~M. Liggett.
\newblock How to find an extra head: optimal random shifts of {B}ernoulli and
  {P}oisson random fields.
\newblock {\em Ann. Probab.}, 29:1405--1425, 2001.

\bibitem{HPPS09}
A.~E. Holroyd, R.~Pemantle, Y.~Peres, and O.~Schramm.
\newblock Poisson matching.
\newblock {\em Ann. Inst. Henri Poincar\'e Probab. Stat.}, 45:266--287, 2009.

\bibitem{HP05}
A.~E. Holroyd and Y.~Peres.
\newblock Extra heads and invariant allocations.
\newblock {\em Ann. Probab.}, 33:31--52, 2005.

\bibitem{LMT14}
G.~Last, P.~M{\"o}rters, and H.~Thorisson.
\newblock Unbiased shifts of {B}rownian motion.
\newblock {\em Ann. Probab.}, 42:431--463, 2014.

\bibitem{LT09}
G.~Last and H.~Thorisson.
\newblock Invariant transports of stationary random measures and
  mass-stationarity.
\newblock {\em Ann. Probab.}, 37:790--813, 2009.

\bibitem{LT02}
T.~M. Liggett.
\newblock Tagged particle distributions or how to choose a head at random.
\newblock In {\em In and out of equilibrium ({M}ambucaba, 2000)}, volume~51 of
  {\em Progr. Probab.}, pages 133--162. Birkh\"auser, Boston, MA, 2002.

\bibitem{P95}
V.~V. Petrov.
\newblock {\em Limit theorems of probability theory}, volume~4 of {\em Oxford
  Studies in Probability}.
\newblock Oxford University Press, New York, 1995.

\bibitem{Re15}
I.~Redl.
\newblock {\em Unbiased shifts of stochastic processes}.
\newblock PhD thesis, University of Bath, forthcoming.

\bibitem{Spitzer60}
F.~Spitzer.
\newblock A {T}auberian theorem and its probability interpretation.
\newblock {\em Trans. Amer. Math. Soc.}, 94:150--169, 1960.

\bibitem{Th00}
H.~Thorisson.
\newblock {\em Coupling, Stationarity, and Regeneration}.
\newblock Probability and its Applications (New York). Springer-Verlag, New
  York, 2000.

\end{thebibliography}

\end{document}